\documentclass[11pt]{amsart}
\usepackage[utf8]{inputenc}
\usepackage{amssymb}
\usepackage{amsfonts}
\usepackage{amsmath}
\usepackage{xcolor}
\usepackage[english]{babel}
\usepackage{amsthm}
\usepackage{setspace}
\usepackage{tikz-cd}
\usepackage{graphicx}
\usepackage[perpage]{footmisc}
\usepackage[shortlabels]{enumitem}
\usepackage{longtable}
\usepackage{comment}
\usepackage{diagbox}
\usepackage{bbm}

\usepackage{mathtools}

\newtheoremstyle{example}
{}                
{}                
{\sffamily}        
{}                
{\bfseries}       
{.}               
{ }               
{}                

\usepackage[
backend=biber,
style=alphabetic
]{biblatex}
\usepackage{biblatex}
\usepackage{hyperref}
\usepackage{xurl}
\hypersetup{breaklinks=true}

\theoremstyle{theorem}
\newtheorem{theorem}{Theorem}

\newtheorem{lemma}[theorem]{Lemma}
\newtheorem{corollary}[theorem]{Corollary}

\theoremstyle{example}

\theoremstyle{definition}
\newtheorem{definition}[theorem]{Definition}

\newtheorem{remark}[theorem]{Remark}

\numberwithin{theorem}{section}

\numberwithin{theorem}{section}

\DeclareMathOperator{\im}{im}

\DeclareMathOperator{\Tr}{Tr}

\DeclareMathOperator{\Gal}{Gal}

\DeclareMathOperator{\Aut}{Aut}

\DeclareMathOperator{\Nm}{Nm}

\DeclareMathOperator{\disc}{disc}

\DeclareMathOperator{\Twist}{Twist}

\newcommand{\co}{\mathcal{O}}

\newcommand{\Z}{\mathbb{Z}}

\newcommand{\Q}{\mathbb{Q}}

\newcommand{\F}{\mathbb{F}}

\newcommand{\p}{\mathfrak{p}}

\newcommand{\m}{\widetilde{m}}
\newcommand{\ext}{\mathrm{ext}}

\newcommand{\quadr}{\mathrm{quad}}

\newcommand{\extendable}{\uparrow C_4}
\newcommand{\Tow}{\mathrm{Tow}}
\newcommand{\oneaut}{\mathrm{1-Aut}}

\usepackage[margin=1.2in]{geometry}
\usepackage[style=alphabetic]{biblatex}
\title[Counting wildly ramified quartic extensions]{Counting wildly ramified quartic extensions with prescribed Galois closure group}
\author{Sebastian Monnet}
\date{}
\begin{document}
\maketitle

\begin{abstract}
    Given a $2$-adic field $K$, we give a formula for the number of totally ramified quartic field extensions $L/K$ with a given discriminant valuation and Galois closure group. We use these formulae to prove refinements of Serre's mass formula, which will have applications to the arithmetic statistics of number fields.   
\end{abstract}
\tableofcontents
\section{Introduction}
\label{sec-intro}

Throughout this paper, we use the term \emph{$2$-adic field} for a finite field extension of the $2$-adic numbers $\Q_2$, and all the fields we consider will be $2$-adic. Once and for all fix a $2$-adic field $K$. 

Let $L/K$ be a finite field extension. The \emph{Galois closure group} of $L/K$ is the Galois group $\Gal(\widetilde{L}/K)$, where $\widetilde{L}$ is the normal closure of $L$ over $K$. Write 
$
\Sigma_m^G
$
for the set of isomorphism classes of totally ramified quartic field extensions $L/K$ with $v_K(d_{L/K}) = m$, such that the Galois closure group of $L/K$ is isomorphic to $G$. We allow ourselves to drop either or both of the decorators $G$ and $m$, with the obvious meanings. 

Using his eponymous lemma, Krasner \cite[Th\'eor\`eme 1]{krasner} found a formula, in terms of $m$, for the size of the set $\Sigma_m$. More recently, Sinclair \cite{sinclair} and Pauli--Sinclair \cite{pauli-sinclair} gave refinements of Krasner's formula, enumerating (among other things) the elements of $\Sigma_m$ that have a prescribed ramification polygon. In a different direction, Wei and Ji \cite{wei-ji} counted the elements of $\Sigma^{S_4}$ and $\Sigma^{A_4}$, without any conditions on discriminant valuation. In this paper, we combine the flavours of \cite{pauli-sinclair} and \cite{wei-ji} to give new refinements of Krasner's result: formulae for the sizes of the sets $\Sigma_m^G$, for all $m$ and $G$. These results in hand, we prove novel refinements of Serre's mass formula, which will have applications in the arithmetic statistics of number fields.

\subsection{Outline and key results}

In Section~\ref{sec-G=1}, we use a result of Serre to relate 
$$
\#\Big(\Sigma_{m}^{S_4} \cup \Sigma_{m}^{A_4}\Big)
$$
to the density of the corresponding Eisenstein polynomials. We then find explicit congruence conditions for this set of Eisenstein polynomials and use them to compute the required density. Finally, we establish conditions for distinguishing between $\Sigma_{m}^{A_4}$ and $\Sigma_{m}^{S_4}$, which we use to obtain the following two results:

\begin{theorem}
    \label{thm-S4-A4-f-even}
    Suppose that $f(K/\Q_2)$ is even. Then $\Sigma_{m}^{S_4}$ is empty for all $m$. Moreover, $\Sigma_{m}^{A_4}$ is nonempty if and only if $m$ is an even integer with $4 \leq m \leq 6e_K + 2$. In that case, we have 
    $$
    \#\Sigma_{m}^{A_4} = \begin{cases}
        \frac{1}{3}q^{\lfloor \frac{m}{3}\rfloor - 2}(q^2-1)\quad\text{if $3 \mid m$},
        \\
        q^{\lfloor \frac{m}{3}\rfloor - 1}(q-1)\quad\text{if $3\nmid m$}.
    \end{cases}
    $$
\end{theorem}
\begin{theorem}
    \label{thm-S4-A4-f-odd}
    Suppose that $f(K/\Q_2)$ is odd. 
    \begin{itemize} 
        \item The set $\Sigma_{m}^{S_4}$ is nonempty if and only if $m \in 2\Z \setminus 6\Z$ and $4 \leq m \leq 6e_K + 2$. In that case, we have
    $$
    \#\Sigma_{m}^{S_4} = q^{\lfloor\frac{m}{3}\rfloor - 1}(q-1). 
    $$
    \item The set $\Sigma_{m}^{A_4}$ is nonempty if and only if $m$ is a multiple of $6$ and $6 \leq m \leq 6e_K$. In that case, we have
    $$
    \#\Sigma_{m}^{A_4} = \frac{1}{3}\cdot q^{\lfloor \frac{m}{3}\rfloor - 2}(q^2-1).
    $$
    \end{itemize}
\end{theorem}
The case $V_4$ was addressed by Tunnell in \cite{tunnell}. We repackage his result in Section~\ref{sec-G=V4} as the following theorem:
\begin{theorem}
    \label{thm-size-of-Sigma-V4-m}
    If $\Sigma_{m}^{V_4}$ is nonempty, then $m$ is an even integer with $6 \leq m \leq 6e_K + 2$. For all such $m$, we have 
    $$
\#\Sigma_{m}^{V_4} = 2(q-1)q^{\frac{m-4}{2}}\Big(
    q^{-\lfloor \frac{m}{6}\rfloor}(1 + \mathbbm{1}_{3 \mid m}\cdot \frac{q-2}{3}) - \mathbbm{1}_{m\leq 4e_K + 2}\cdot q^{-\lfloor\frac{m-2}{4}\rfloor}
\Big).
$$
\end{theorem}
The bulk of our work goes into the $C_4$ case. In \cite{cohen-et-al}, Cohen, Diaz y Diaz, and Olivier obtain asymptotic formulae for the number of $C_4$-extensions of a number field. We adapt their methods to compute the size of $\Sigma_{m}^{C_4}$. Our formula depends on the discriminant valuation
$$
d_{(-1)} = v_K(d_{K(\sqrt{-1})/K}),
$$
which is an even integer by Lemma~\ref{lem-bound-on-d-(-1)}. 
\begin{theorem}
    \label{thm-size-Sigma-C4-m}
    If $\Sigma_{m}^{C_4}$ is nonempty, then either $m=8e_K + 3$ or $m$ is an even integer with $8 \leq m \leq 8e_K$. For even $m$ with $8 \leq m \leq 8e_K$, the number $\# \Sigma_{m}^{C_4}$ is the sum of the following four quantities:
    \begin{enumerate}
        \item $\mathbbm{1}_{8\leq m \leq 5e_K - 2}\cdot \mathbbm{1}_{m\equiv 3\pmod{5}}\cdot 2q^{\frac{3m-14}{10}}(q - 1)$.
        \item $\mathbbm{1}_{4e_K + 4 \leq m \leq 5e_K + 2}\cdot 2q^{\frac{m}{2}-e_K - 2}(q - 1)$.
        \item $\mathbbm{1}_{5e_K + 3 \leq m \leq 8e_K}\cdot \mathbbm{1}_{m\equiv 2e_K\pmod{3}}\cdot 2q^{\frac{m+4e_K}{6} - 1}(1 + \mathbbm{1}_{m \leq 8e_K - 3d_{(-1)}})(q - 1 - \mathbbm{1}_{m = 8e_K - 3d_{(-1)} + 6})$. 
        \item $\mathbbm{1}_{10 \leq m \leq 5e_K}\cdot 2(q-1)(
            q^{ \lfloor\frac{3m}{10} \rfloor - 1} - q^{\max\{
                \lceil \frac{m+2}{4}\rceil,
                 \frac{m}{2} - e_K
            \} - 2}
        )$.
    \end{enumerate}
    We also have 
    $$
    \#\Sigma_{8e_K + 3}^{C_4} = \begin{cases}
        4q^{2e_K}\quad\text{if $-1 \in K^{\times 2}$},
        \\
        2q^{2e_K} \quad\text{if $K(\sqrt{-1})/K$ is quadratic and totally ramified},
        \\
        0\quad\text{if $K(\sqrt{-1})/K$ is quadratic and unramified}.
    \end{cases}
    $$
\end{theorem}
Finally, in Section~\ref{sec-G=D4}, we compute the number of towers of two quadratic extensions $L/E/K$ with $v_K(d_{L/K}) = m$ and express this number in terms of $\#\Sigma_{m}^{C_4}$, $\#\Sigma_{m}^{V_4}$, and $\#\Sigma_{m}^{D_4}$. Rearranging, we obtain:
\begin{theorem}
    \label{thm-size-of-Sigma_m_D4}
    If $\Sigma_{m}^{D_4}$ is nonempty, then one of the following holds:
    \begin{enumerate}
        \item $m$ is an even integer with $6  \leq m \leq 8e_K + 2$.
        \item $m\equiv 1 \pmod{4}$ and $4e_K + 5 \leq m \leq 8e_K + 1$. 
        \item $m = 8e_K + 3$. 
    \end{enumerate} 
    For even $m$ with $6 \leq m \leq 8e_K + 2$, we have 
    $$
    \#\Sigma_{m}^{D_4} = 2(q-1)q^{\frac{m}{2} - 2}(\mathbbm{1}_{m \geq 4e_K + 4}\cdot q^{-e_K} +\mathbbm{1}_{m \leq 8e_K}\cdot ( q^{\min\{0, e_K + 1 - \lceil \frac{m}{4}\rceil\} }- q^{-\min\{\lfloor \frac{m-2}{4}\rfloor, e_K\}})) - \frac{1}{2}\#\Sigma_m^{C_4} - \frac{3}{2}\#\Sigma_m^{V_4}. 
    $$
    For $m\equiv 1\pmod{4}$ with $4e_K + 5 \leq m \leq 8e_K + 1$, we have 
    $$
    \#\Sigma_{m}^{D_4} = 2(q-1)q^{e_K + \frac{m-1}{4} - 1} - \frac{1}{2}\#\Sigma_m^{C_4} - \frac{3}{2}\#\Sigma_m^{V_4}. 
    $$
    If $m = 8e_K + 3$, then 
    $$
    \#\Sigma_{m}^{D_4} = 2q^{3e_K} - \frac{1}{2}\#\Sigma_{8e_K + 3}^{C_4}.
    $$
    Theorems~\ref{thm-size-of-Sigma-V4-m} and \ref{thm-size-Sigma-C4-m} make these expressions completely explicit.
\end{theorem}

\subsection{Application: Refinements of Serre's mass formula}

Our main application is to prove refinements of Serre's mass formula. Define the \emph{mass} of a set $S$ of field extensions $L/K$ to be 
$$
\m(S) = \sum_{L \in S} \frac{(\#\Aut(L/K))^{-1}}{q^{v_K(d_{L/K})}}.
$$
This quantity was first studied by Serre, who proved his famous ``mass formula'' \cite[Theorem~2]{serre}. In \cite{bhargava-serre-mass-formula}, Bhargava generalised Serre's formula to sets of \'etale algebras over $K$ and developed the so-called ``Malle--Bhargava heuristics'' which predict the asymptotic number of degree $n$ number fields with Galois closure group $S_n$, when ordered by discriminant. Essentially, Bhargava predicts that the probability of a ``randomly selected'' such number field having a prescribed local completion is proportional to the mass of that local completion. 

Bhargava, Shankar, and Wang proved these heuristics for $n=2,3,4,5$ in \cite[Theorem~2]{bhargava-shankar-wang}, replacing degree $n$ number fields by degree $n$ extensions of an arbitrary base number field. Recently, in \cite{alberts2023random}, Alberts extended the Malle--Bhargava heuristics, replacing $S_n$ with more general classes of Galois closure groups. Aside from the mass's general importance in arithmetic statistics, the original motivation for our refinements comes from our earlier preprint \cite{monnet-S4-quartics}; our formula for the local masses at primes $\p$ lying over $2$ (called $m_{\mathcal{A},\p}$ in \cite{monnet-S4-quartics}) is woefully inexplicit, and we intend to use the formulae in this paper to remedy that shortcoming. Similarly, upcoming work of Newton--Varma uses a modified version of Corollary~\ref{cor-premass-of-Sigma-C4}, and more generally we expect our refined mass formulae to be useful for obtaining explicit masses when counting $S_4$-quartic extensions with local conditions.  

We find explicit formulae for $\m(\Sigma^{G})$ for each $G$, which we now state. The proofs are deferred to later sections of the paper.

\begin{corollary}
    \label{cor-premass-of-S4-A4-f-even}
    If $f(K/\Q_2)$ is even, then 
    $$
    \m(\Sigma^{S_4}) = 0,
    $$
    and
    $$
    \m(\Sigma^{A_4}) = 
    \frac{1}{3}(q-1)\cdot\frac{q^{4e_K}-1}{q^4-1}\cdot q^{-4e_K-3}\Big(
        3q^3 + q^2 + q + 3
    \Big).
    $$
\end{corollary}
\begin{corollary}
    \label{cor-premass-of-S4-A4-f-odd}
    Suppose that $f(K/\Q_2)$ is odd. Then 
    $$
    \m(\Sigma^{S_4}) = \frac{q^3 + 1}{q^3 + q^2 + q + 1}\cdot (q^{-3} - q^{-4e_K - 3}),
    $$
    and 
    $$
    \m(\Sigma^{A_4}) = \frac{1}{3}\cdot \frac{1}{q^2 + 1}\cdot (q^{-2} - q^{-4e_K - 2}).
    $$
\end{corollary}
\begin{corollary}
    \label{cor-premass-C2xC2}
    We have 
    $$
    \m(\Sigma^{V_4}) = \frac{q-1}{6}\cdot \Big(q^{-4e_K-3}\cdot \frac{q^{4e_K} - 1}{q^4 - 1} \cdot (3q^3 + q^2 + q + 3) - 3 q^{-3e_K - 3}\cdot \frac{q^{3e_K} - 1}{q^3 - 1} \cdot (q^2 + 1)\Big).
    $$
\end{corollary}
\begin{corollary}
    \label{cor-premass-of-Sigma-C4}
    The mass 
    $
    \widetilde{m}(\Sigma^{C_4})
    $
    is the sum of the following nine quantities:
    \begin{enumerate}
        \item $$
        \frac{1}{2}\cdot \frac{(q-1)(1 - q^{-7\lfloor \frac{e_K}{2}\rfloor})}{q^7 - 1}.
        $$
        \item $$
        \frac{1}{2}\cdot q^{-3e_K - 3}(1 - q^{
                - \lfloor \frac{e_K}{2}\rfloor
            }).
        $$
        \item $$
        \mathbbm{1}_{d_{(-1)} < e_K}\cdot \frac{(q-1)(q^{-5\lfloor\frac{e_K}{2}\rfloor - e_K - 1} - q^{\frac{5}{2} d_{(-1)} - 6e_K - 1})}{q^5-1}.
        $$
        \item $$
        \frac{1}{2}\cdot \mathbbm{1}_{d_{(-1)} \geq 2}\cdot q^{-6e_K + \frac{5}{2}d_{(-1)} - 6}(q-2).     
    $$
    \item $$
    \frac{1}{2} \cdot \mathbbm{1}_{d_{(-1)}\geq 4}\cdot \frac{(q-1)(q^{\frac{5}{2}d_{(-1)} - 6e_K - 6} - q^{-6e_K-1})}{q^5-1}.
    $$
    \item $$\mathbbm{1}_{e_K\geq 2}\cdot \frac{1}{2}(q-1)q^{- 7\lfloor \frac{e_K}{2}\rfloor - 1}\Big(
        \frac{q(q^{7\lfloor \frac{e_K}{2}\rfloor - 7} - 1)(q^6 + q^4 + q^3 + q + 1)}{q^7-1} + 1 + \mathbbm{1}_{2\nmid e_K}(q^{-2} + q^{-3})    
    \Big).$$
    \item $$
    -\mathbbm{1}_{e_K \geq 2}\cdot \frac{1}{2}\cdot\frac{(q-1)(q+1)(q^{-7} - q^{-3e_K - 1})}{q^3-1}.
    $$
    \item $$
    - \frac{1}{2}q^{-3e_K - 2}(1 - q^{-\lfloor\frac{e_K}{2}\rfloor}).
    $$
    \item $$
        \begin{cases}
            q^{-6e_K - 3} \quad\text{if $-1 \in K^{\times 2}$},
            \\
            \frac{1}{2}q^{-6e_K-3}\quad\text{if $K(\sqrt{-1})/K$ is quadratic and totally ramified},
            \\
            0 \quad\text{otherwise}. 
        \end{cases}
    $$
    \end{enumerate}
\end{corollary}

\begin{corollary}
    \label{cor-premass-of-D4}
    We have the following formula for $\m(\Sigma^{D_4})$, which is made completely explicit by Corollaries~\ref{cor-premass-C2xC2} and \ref{cor-premass-of-Sigma-C4}. 
    $$
    \m(\Sigma^{D_4}) = \frac{1}{q^2+q+1} \cdot (q^{-3e_K - 3} + q^{-3e_K - 1} + q^{-2}) - \m(\Sigma^{C_4}) - 3\m(\Sigma^{V_4}).
    $$
\end{corollary}

\subsection{Correctness of results}
Using MAGMA \cite{magma} and the LMFDB \cite{lmfdb}, we have verified Theorems~\ref{thm-S4-A4-f-even}-\ref{thm-size-of-Sigma_m_D4} and Corollaries~\ref{cor-premass-of-S4-A4-f-even}-\ref{cor-premass-of-D4} for all extensions $K/\Q_2$ of degree at most $3$. Whenever $e_K\leq 10$ and $f_K \leq 10$, we have also checked numerically the deduction of Corollaries~\ref{cor-premass-of-S4-A4-f-even}-\ref{cor-premass-of-D4} from Theorems~\ref{thm-S4-A4-f-even}-\ref{thm-size-of-Sigma_m_D4}. The proofs involve a lot of messy computation, so we are reassured to have checked our work numerically. Our code is available at \url{https://github.com/Sebastian-Monnet/Mass-Formula-Checks}.

\subsection{Index of notation}
We fix the following notation. 
\begin{enumerate}
    \item For a $2$-adic field $F$, write:
    \begin{enumerate}
        \item $\co_F$ for its ring of integers.
        \item $\pi_F$ for a uniformiser of $\co_F$. 
        \item $\p_F$ for the maximal ideal of $\co_F$. 
        \item $\F_F$ for the residue field $\co_F/\p_F$.
        \item $q_F$ for the cardinality of $\F_F$.
        \item $e_F$ for the absolute ramification index $e(F/\Q_2)$.
        \item $f_F$ for the inertia degree $f(F/\Q_2)$.
        \item $v_F$ for the unique $2$-adic valuation on $F$, normalised such that $v_F(\pi_F) = 1$.
        \item $U_F^{(i)}$ for the group $1 + \p_F^i$ in the unit filtration, where $i > 0$. 
        \item $U_F^{(0)}$ for the unit group $\co_F^\times$. 
    \end{enumerate}
    \item Given an extension $E/F$ of $2$-adic fields, write: 
    \begin{enumerate}
        \item $d_{E/F}$ for its discriminant ideal.
        \item $e(E/F)$ for its ramification index. 
        \item $f(E/F)$ for its inertia degree. 
    \end{enumerate}
    \item $K$: A fixed $2$-adic field.
    \item $q$: Shorthand for $q_K$. 
    \item $\Sigma$: The set of totally ramified quartic extensions $L/K$.
    \item $\widetilde{L}$: For an extension $L/K$ of $K$, write $\widetilde{L}$ for the normal closure of $L$ over $K$. 
    \item For a group $G \in \{S_4, A_4, D_4, V_4, C_4\}$ and positive integer $m$, write:
    \begin{enumerate}
        \item $\Sigma^G := \{L \in \Sigma : \Gal(\widetilde{L}/K) \cong G\}$.
        \item $\Sigma_m := \{L \in \Sigma : v_K(d_{L/K}) = m\}$.
        \item $\Sigma_m^G := \Sigma_m \cap \Sigma^G$. 
    \end{enumerate}
    \item $d_{(-1)}$: The discriminant valuation $v_K(d_{K(\sqrt{-1})/K})$.
    \item $\m(S)$: For $S \subseteq \Sigma$, the \emph{mass} of $S$ is $$\m(S) = \sum_{L \in S} \frac{(\#\Aut(L/K))^{-1}}{q^{v_K(d_{L/K})}}.$$ 
    \item $\Sigma^{\oneaut}$ and $\Sigma^{\oneaut}_m$: The sets $\Sigma^{A_4}\cup \Sigma^{S_4}$ and $\Sigma^{\oneaut}\cap \Sigma_m$ respectively. 
    \item $P$: The set of monic, quartic Eisenstein polynomials in $K[X]$. 
    \item $L_f$: For $f \in P$, write $L_f$ for the field extension $K[X]/(f)$ of $K$.
    \item $P_m, P^G$ and $P^G_m$: For $G \in \{S_4, A_4, D_4, V_4, C_4\}$, define $P_m, P^G$, and $P^G_m$ to be the sets of $f \in L$ such that $L_f$ is in $\Sigma_m, \Sigma^G$, and $\Sigma^G_m$ respectively.
    \item $P^{\oneaut}$ and $P_m^{\oneaut}$: The sets of $f \in P$ such that $L_f \in \Sigma^{\oneaut}$ and $L_f \in \Sigma_m^{\oneaut}$, respectively.
    \item $\mu$: Haar measure on $\co_K^4$, normalised so that $\mu(\co_K^4) = 1$. 
    \item $v_\pi$: Given an extension $L/K$, such that $\pi$ is a uniformiser of $L$, we write $v_\pi$ for the $2$-adic valuation on $L$, normalised so that $v_\pi(\pi) = 1$.
    \item $T_m$: For even integers $4 \leq m \leq 6e_K + 2$, this is the set of $a_0 + a_1 X + a_2X^2 + a_3X^3 + X^4$ in $P$ such that 
    $$
    \begin{cases}
        v_K(a_1)= \frac{m}{4}, \quad v_K(a_2)\geq \frac{m}{6},\quad v_K(a_3) \geq \frac{m}{4}, \quad \text{if $m\equiv 0 \pmod{4}$},
        \\
        v_K(a_1) \geq \frac{m+2}{4},\quad v_K(a_2)\geq \frac{m}{6},\quad v_K(a_3) = \frac{m-2}{4},\quad \text{if $m \equiv 2\pmod{4}$}. 
    \end{cases}
    $$
    \item $\mathcal{R}$: System of representatives for $(\co_K/\p_K)^\times$.
    \item $g_f^{(u)}$: When $m$ is a multiple of $6$ with $6 \leq m \leq 6e_K$, $u \in \mathcal{R}$, and $f \in P_m$, we define $g_f^{(u)}(X) = (X + \pi + u\pi^{\frac{m}{3}})$. 
    \item $b_i^{(u)}$: The $X^i$ coefficient of $g_f^{(u)}$. 
    \item $\mu_3 \subseteq K$: Shorthand for ``$K$ contains three distinct cube roots of unity''. 
    \item Let $F$ be a $2$-adic field. Write:
    \begin{enumerate}
    \item $\Sigma_{\quadr/F}$ for the set of isomorphism classes of quadratic extensions of $F$. 
    \item $\Sigma_{\quadr/F,m}$ (respectively $\Sigma_{\quadr/F,\leq m}$) for the set of $E \in \Sigma_{\quadr/F}$ such that we have $v_F(d_{E/F}) = m$ (respectively $v_F(d_{E/F}) \leq m$).
    \end{enumerate}
    \item Let $E/K$ be a quadratic extension. Write: 
    \begin{enumerate}
    \item $\Sigma_{\quadr/E}^{G/K}$ for the set of $L \in \Sigma_{\quadr/E}$ such that the Galois closure group of $L/K$ is isomorphic to $G$. 
    \item $\Sigma_{\quadr/E,m_2}^{G/K}$ (respectively $\Sigma_{\quadr/E, \leq m_2}^{G/K}$) for the set of $L \in \Sigma_{\quadr/E}^{G/K}$ such that $v_E(d_{L/E}) = m_2$ (respectively $v_E(d_{L/E}) \leq m_2$). 
    \end{enumerate}
    \item $\Sigma_{\quadr/K,m_1}^{\uparrow C_4}$ (respectively $\Sigma_{\quadr/K,\leq m_1}^{\uparrow C_4}$): Set of $C_4$-extendable quadratic extensions $E/K$ such that $v_K(d_{E/K}) = m_1$ (respectively $v_K(d_{E/K}) \leq m_1$).   
    \item $N_\ext(m_1)$: Function explicitly defined in Definition~\ref{defi-N-ext}.
    \item $N^{C_4}(m_1,m_2)$: Function explicitly defined in Definition~\ref{defi-N-C4}.
    \item $S_{F/K, t}$: For an extension $F/K$ and an integer $1 \leq t \leq v_F(2)$, we define 
    $$
    S_{F/K,t} = \{u \in K^\times / K^{\times 2} : u/x^2 \equiv 1 \pmod{\p_F^{2t}} \text{ for some $x \in F^\times$}\}.
    $$
    For $t=0$, define $$S_{F/K,0} = \{u \in K^\times / K^{\times 2} : v_F(u) \text{ is even}\}.$$
    \item Let $\mathcal{A}$ be a subgroup of $K^\times / K^{\times 2}$. Then write: 
    \begin{enumerate}
    \item $K(\sqrt{\mathcal{A}}) = K(\{\sqrt{\alpha} : [\alpha] \in \mathcal{A}\})$. 
    \item $\co_K^\mathcal{A} = \co_K^\times \cap \Nm K(\sqrt{\mathcal{A}})$. 
    \item $S_{K/K,t}^\mathcal{A} = S_{K/K,t} \cap \big(\Nm K(\sqrt{\mathcal{A}})/K^{\times 2}\big)$, where $0 \leq t \leq e_K$.
    \item $\Sigma_{\quadr/K, \leq m_1}^\mathcal{A}$ for the set of $E \in \Sigma_{\quadr/K, \leq m_1}$ with $\mathcal{A}\subseteq \Nm E$. 
    \item $(\co_K/\p_K^{2t})^\mathcal{A}$ for the image of the map $\co_K^\mathcal{A} \to (\co_K/\p_K^{2t})^{\times}$.
    \item $\mathcal{A}_t = \mathcal{A} \cap \big(U_K^{(2t)}K^{\times 2} / K^{\times 2}\big)$, where $0 \leq t \leq e_K$. 
    \end{enumerate}
    \item $\Sigma_{\quadr/K}^{\hookrightarrow L}$: For $G\in \{V_4,C_4,D_4\}$ and $L \in \Sigma^G$, this is the set of $E \in \Sigma_{\quadr/K}$ such that there exists a $K$-morphism $E \hookrightarrow L$. 
    \item $\Twist_K(L/E)$: For $G \in \{V_4, C_4, D_4\}$, $L \in \Sigma^G$, and $E \in \Sigma_{\quadr/K}^{\hookrightarrow L}$, this is the set of $L' \in \Sigma_{\quadr/E}$ such that there is a $K$-isomorphism $L\to L'$, where $L$ is viewed as an extension of $E$ via the unique embedding $E\hookrightarrow L$. 
    \item $\Tow_m$: The set of pairs $(E,L)$, where $L/E/K$ is a tower of totally ramified quadratic extensions and $v_K(d_{L/K}) = m$.
    \item $\Phi_m$: The forgetful map $\Tow_m \to \Sigma_m$, taking $(E,L)$ to the extension $L$ of $K$. 
\end{enumerate}

\subsection{Acknowledgements}
I am grateful to my supervisor, Rachel Newton, for her support and enthusiasm throughout the project. Thanks also to Melanie Matchett Wood and Takehiko Yasuda for helpful suggestions, especially in the Galois cases, and to John Voight for suggesting I use the LMFDB to check my results. I am particularly indebted to Lee Berry, Ross Paterson, and Tim Santens for some very helpful conversations. 

This work was supported by the Engineering and Physical Sciences Research Council [EP/S021590/1], via the EPSRC Centre for Doctoral Training in Geometry and Number Theory (The London School of Geometry and Number Theory), University College London.

\section{The cases $G = S_4$ and $G = A_4$}
\label{sec-G=1}

Throughout this paper, all Eisenstein polynomials are taken to be monic. Write $P$ for the set of quartic Eisenstein polynomials in $K[X]$. For $f \in P$, let $L_f$ be the field $K[X]/(f)$, which is a totally ramified quartic extension of $K$. Given a finite group $G$, let $P^G$ be the set of $f \in P$ such that $L_f/K$ has Galois closure group isomorphic to $G$. For any integer $m$, let $P_m$ be the set of $f \in P$ such that $v_K(d_{L_f/K}) = m$, or equivalently such that $v_K(\disc(f)) = m$. For each $G$, write $P^G_m$ for the intersection $P^G \cap P_m$. Write $P^{\oneaut}$ and $P_m^{\oneaut}$ as shorthand\footnote{The $\oneaut$ refers to the fact that $\#\Aut(L/K) = 1$ if and only if $L \in \Sigma^{S_4}\cup \Sigma^{A_4}$.} for $P^{S_4}\cup P^{A_4}$ and $P_m^{S_4} \cup P_m^{A_4}$ respectively. Similarly, write $\Sigma^{\oneaut}$ and $\Sigma_m^{\oneaut}$ for $\Sigma^{S_4}\cup \Sigma^{A_4}$ and $\Sigma_m^{S_4} \cup \Sigma_m^{A_4}$ respectively.

The quartic Eisenstein polynomials in $K[X]$ embed naturally into $\co_K^4$ via 
$$
X^4 + a_3X^3 + a_2X^2 + a_1X + a_0 \mapsto (a_3,a_2,a_1,a_0).
$$
Write $\mu$ for the Haar measure on $\co_K^4$, normalised such that $\mu(\co_K^4) = 1$. We will apply this Haar measure to sets of Eisenstein polynomials, viewed as subsets of $\co_K^4$ via the embedding described above. 

\begin{lemma}
    \label{lem-num-exts-density-polys}
    Let $G \in \{S_4, A_4\}$ let $m$ be a positive integer. We have 
    $$
    \#\Sigma_m^G = \frac{q^{m+2}}{q-1}\cdot \mu(P^G_m). 
    $$
\end{lemma}
\begin{proof}
    This follows easily from \cite[Equation~13]{serre}.
\end{proof}

\subsection{Congruence conditions for $P^{\oneaut}_m$}

In \cite[Theorem~2.9]{lbekkouri}, Lbekkouri gives congruence conditions for a quartic Eisenstein polynomial $f(X) \in \Q_2[X]$ to define a Galois extension. We extend his methods to Eisenstein polynomials over arbitrary $2$-adic base fields, to obtain congruence conditions for the set $P_m^{\oneaut}$, which we will state in Lemma~\ref{lem-Pm1-eq-Tm-if-m-not-multiple-of-3} and Corollary~\ref{cor-cong-cond-m-multiple-of-6}.
\begin{remark}
    It should be noted that Lbekkouri's statement of \cite[Theorem~2.9]{lbekkouri} is incorrect. In items (2i) and (2ii), both instances of ``$a_0 + a_2$'' should read ``$a_0 + 2$''. This typo is first introduced in the statement of Proposition~2.8 and is carried over into Theorem~2.9. 
\end{remark}
\begin{remark}
    We tried to extend the method to find congruence conditions for $L_f/K$ to be Galois. This was a partial success, and we came up with an algorithm that can compute such congruence conditions given a choice of base field $K$. Our results are not organised enough to be publishable, but the interested reader should get in touch! 
\end{remark}

For $f \in P$, we will always denote the coefficients of $f$ by $f(X) = X^4 + a_3X^3 + a_2X^2 + a_1X + a_0$. Whenever we refer to the coefficients $a_i$, the choice of $f$ will be clear. Let $\pi_f = X + (f)$ be the natural uniformiser of $L_f$. We will always drop the subscript and denote $\pi_f$ by $\pi$, since our choice of $f$ will be clear. Write $v_\pi$ for the $2$-adic valuation on $L_f$, normalised such that $v_\pi(\pi) = 1$. Fix an algebraic closure $\overline{K}$ of $L_f$, and let 
$$
\sigma_i : L_f \to \overline{K},\quad i=1,2,3,4
$$
be the four embeddings of $L_f$, where $\sigma_1$ is the identity embedding. 
\begin{lemma}
    \label{lem-1-aut-implies-eq-vals}
    For all $f \in P^{\oneaut}$, the three valuations 
    $$
    v_K(\sigma_i(\pi) - \pi),\quad i=2,3,4
    $$
    are equal. 
\end{lemma}
\begin{proof}
    Suppose that $f \in P$ and the quantities $v_K(\sigma_i(\pi) - \pi)$ are not all equal for $i=2,3,4$. Reordering the $\sigma_i$ if necessary, we have 
    $$
    v_K(\sigma_2(\pi) - \pi) \neq v_K(\sigma_i(\pi) - \pi)
    $$
    for $i = 3,4$. The cubic polynomial $X^{-1}f(X+\pi) \in L_f[X]$ has roots
    $$
    \sigma_i(\pi) - \pi,\quad i = 2,3,4.
    $$
    The minimal polynomial of $\sigma_2(\pi) - \pi$ over $L_f$ therefore divides $X^{-1}f(X+\pi)$, and all its roots have the same valuation, so 
    $$
    \sigma_2(\pi) - \pi \in L_f,
    $$
    and therefore $f$ has at least two roots in $L_f$, so $f \not \in P^{\oneaut}$.
\end{proof}

For each even integer $4 \leq m \leq 6e_K + 2$, define $T_m$ to be the set of $f \in P$ such that 
$$
\begin{cases}
    v_K(a_1)= \frac{m}{4}, \quad v_K(a_2)\geq \frac{m}{6},\quad v_K(a_3) \geq \frac{m}{4}, \quad \text{if $m\equiv 0 \pmod{4}$},
    \\
    v_K(a_1) \geq \frac{m+2}{4},\quad v_K(a_2)\geq \frac{m}{6},\quad v_K(a_3) = \frac{m-2}{4},\quad \text{if $m \equiv 2\pmod{4}$}. 
\end{cases}
$$
\begin{lemma}
    \label{lem-T-m-iff-val-m-by-12}
     The following two statements are true:
    \begin{enumerate} 
    \item  Let $m$ be an even integer with $4 \leq m \leq 6e_K + 2$ and let $f \in P_m$. Then $f \in T_m$ if and only if 
    $$
    v_K(\sigma_i(\pi) - \pi) = \frac{m}{12}
    $$
    for $i=2,3,4$.
    \item Let $m$ be a positive integer. If $P_m^{\oneaut}$ is nonempty then $m$ is even, $4 \leq m \leq 6e_K + 2$, and $P_m^{\oneaut}\subseteq T_m$. 
    \end{enumerate}
\end{lemma}
\begin{proof}
    Let $f \in P_m$ for any positive integer $m$, not necessarily even. Define the polynomial 
    $$
    g(X) := X^{-1}f(X + \pi),
    $$
    and write $g(X) = \sum_{i=0}^3 b_iX^i$ for $b_i \in L_f$. 
    It is easy to see that 
    \begin{align*}
        b_0 &= a_1 + 2\pi a_2 + 3\pi^2 a_3 + 4\pi^3,
        \\
        b_1 &= a_2 + 3\pi a_3 + 6\pi^2,
        \\
        b_2 &= a_3 + 4\pi. 
    \end{align*}
    Since the $v_\pi(a_i)$ are all multiples of $4$, we have 
    \begin{align*}
        v_\pi(b_0) &= \min\{v_\pi(a_1), v_\pi(2\pi a_2), v_\pi(3\pi^2 a_3), v_\pi(4\pi^3)\},
        \\
        v_\pi(b_1) &= \min\{v_\pi(a_2), v_\pi(3\pi a_3), v_\pi(6\pi^2)\},
        \\
        v_\pi(b_2) &= \min\{v_\pi(a_3), v_\pi(4\pi)\}.
    \end{align*}
    The polynomial $g(X) \in L_f[X]$ has roots $\sigma_i(\pi) - \pi$ for $i=2,3,4$. Suppose that
    $$
    v_K(\sigma_i(\pi) - \pi) = \frac{m}{12}
    $$ 
    for each $i$. Then the Newton polygon of $g(X)$ consists of one line segment $(0,m)\leftrightarrow(3,0)$, so
    \begin{equation}
        \label{eq-m-vals-of-ai}\tag{$*$}
    \begin{cases}
        m = \min\{v_\pi(a_1), v_\pi(2\pi a_2), v_\pi(3\pi^2 a_3), v_\pi(4\pi^3)\},
        \\
        \frac{2m}{3} \leq \min\{v_\pi(a_2), v_\pi(3\pi a_3), v_\pi(6\pi^2)\},
        \\
        \frac{m}{3} \leq \min\{v_\pi(a_3), v_\pi(4\pi)\},
\end{cases}
\end{equation}
    and for even $m$ this implies membership of $T_m$. Reversing the argument, it is easy to see that for even $m$ with $4 \leq m \leq 6e_K + 2$, every $f \in T_m$ has 
    $$
    v_K(\sigma_i(\pi) - \pi) = \frac{m}{12},\quad i=2,3,4.
    $$
    Thus we have proven (1). Now, let $f \in P_m^{\oneaut}$ for some positive integer $m$. Then (\ref{eq-m-vals-of-ai}) tells us that 
    \begin{align*}
        m &= \min\{v_\pi(a_1), v_\pi(a_2) + 4e_K + 1, v_\pi(a_3) + 2, 8e_K + 3\},
        \\
        \frac{2m}{3} &\leq \min\{v_\pi(a_2), 4e_K + 2\}.
    \end{align*}
    Since $f$ is Eisenstein, $v_\pi(a_i)\geq 4$ for each $i$, and therefore we obtain $4 \leq m \leq 6e_K + 3$. Moreover, $v_\pi(a_2) \geq \frac{2m}{3}$ implies that $v_\pi(a_2) + 4e_K + 1 > m$. Since $m < 8e_K + 3$, we obtain
    $$
    m = \min\{v_\pi(a_1), v_\pi(a_3) + 2\},
    $$
    so $m$ is even. Finally, Lemma~\ref{lem-1-aut-implies-eq-vals} tells us that the valuations $v_K(\sigma_i(\pi) - \pi)$ are all equal, so 
    $$
    v_K(\sigma_i(\pi) - \pi) = \frac{m}{12}
    $$
    by definition of discriminant, and therefore $f \in T_m$ by (1). 
\end{proof}

\begin{lemma}
    \label{lem-Pm1-eq-Tm-if-m-not-multiple-of-3}
    Let $m$ be an even integer with $4 \leq m \leq 6e_K + 2$. If $m$ is not a multiple of $3$, then $P_m^{\oneaut} = T_m$. 
\end{lemma}
\begin{proof}
    Lemma~\ref{lem-T-m-iff-val-m-by-12} tells us that $P_m^{\oneaut} \subseteq T_m$, so we just need to show that $T_m \subseteq P_m^{\oneaut}$. Let $f \in T_m$. Lemma~\ref{lem-T-m-iff-val-m-by-12} tells us that 
    $$
    v_\pi(\sigma_i(\pi) - \pi) = \frac{m}{3},\quad i=2,3,4,
    $$
    so $\sigma_i(\pi) \not \in L_f$ for each $i$, since $\frac{m}{3}$ is not an integer, and therefore $T_m \subseteq P_m^{\oneaut}$. 
\end{proof}

From now on, fix a system of representatives $\mathcal{R}$ for $(\co_K/\p_K)^\times$. When $3 \mid m$, for each $u \in \mathcal{R}$ and $f \in P_m$, define the polynomial 
$$
g_f^{(u)}(X) := f(X + \pi + u\pi^{\frac{m}{3}}),
$$
and write $g_f^{(u)}(X) = \sum_{i=0}^4 b_i^{(u)}X^i$ for $b_i^{(u)} \in L_f$. We will always omit the subscript and write $g^{(u)}(X)$ for $g_f^{(u)}(X)$, leaving $f$ implicit. 

\begin{lemma}
    \label{lem-vals-of-bi}
    Let $m$ be a multiple of $6$ with $4 \leq m \leq 6e_K + 2$. Let $f \in T_m$ and $u\in\mathcal{R}$. We have: 
    \begin{enumerate} 
        \item $v_K(b_3^{(u)}) \geq \frac{m-2}{4}$.
        \item $v_K(b_2^{(u)}) \geq \frac{m}{6}$.
        \item $v_K(b_1^{(u)}) = \frac{m}{4}$. 
        \item $$
    v_K(b_0^{(u)}) \begin{cases}
        \geq \frac{m}{3} + 1 \quad\text{if $4\mid m$ and $a_1 + ua_2a_0^{ \frac{m}{12}} + u^3 a_0^{ \frac{m}{4}}\equiv 0 \pmod{\p_K^{ \frac{m}{4} + 1}}$},
        \\
        \geq \frac{m}{3} + 1 \quad\text{if $4\nmid m$ and $a_3 + ua_2a_0^{\lfloor \frac{m}{12}\rfloor} + u^3a_0^{\lfloor \frac{m}{4}\rfloor}\equiv 0\pmod{\p_K^{\lfloor \frac{m}{4}\rfloor + 1}}$},
        \\
        = \frac{m}{3} \quad\text{otherwise}. 
    \end{cases}
    $$
\end{enumerate}
\end{lemma}

\begin{proof}
    It is easy to see that for each $i$ and $u$, we have 
    $$
    b_i^{(u)} = \sum_{j=i}^{4}\binom{j}{i}a_j(\pi + u\pi^{\frac{m}{3}})^{j-i},
    $$
    where we adopt the convention that $a_4 = 1$. Using this formula for the $b_i^{(u)}$, along with the congruence conditions defining $T_m$, gives us the following three congruences:
    \begin{align*} 
    b_3^{(u)}&\equiv a_3 \pmod{\pi^{m + 1}}.
    \\
    b_2^{(u)} &\equiv a_2\pmod{\pi^{\frac{2m}{3} + 1}}.
     \\
    b_1^{(u)} &\equiv \begin{cases} a_1\pmod{\pi^{m+1}}\quad\text{if $m\equiv 0\pmod{4}$},
        \\
        3\pi^2a_3\pmod{\pi^{m + 1}}\quad\text{if $m\equiv 2\pmod{4}$}.  
    \end{cases}  
\end{align*}
    We can read off the first three claims from these congruences. Expanding the formula for $b_0^{(u)}$ and ignoring the high-valuation terms, we obtain 
    $$
    b_0^{(u)} \equiv \begin{cases}
        ua_1  \pi^{\frac{m}{3}} + u^2a_2  \pi^{\frac{2m}{3}} + u^4 \pi^{\frac{4m}{3}} \pmod{\pi^{\frac{4m}{3}+1}} \quad\text{if $m\equiv 0\pmod{4}$},
        \\
        u^2 a_2\pi^{\frac{2m}{3}} + ua_3 \pi^{\frac{m}{3} + 2} + u^4\pi^{\frac{4m}{3}} \pmod{\pi^{\frac{4m}{3} + 1}}\quad\text{if $m\equiv 2\pmod{4}$}.
    \end{cases}
    $$
    It follows that $v_K(b_0^{(u)}) \geq \frac{m}{3}$, and $v_K(b_0^{(u)}) \geq \frac{m}{3}+1$ if and only if 
    $$
    \begin{cases}
        a_1 + ua_2\pi^{\frac{m}{3}} + u^3\pi^m \equiv 0\pmod{\pi^{m+1}}\quad\text{if $m\equiv 0\pmod{4}$},
        \\
        a_3 + ua_2\pi^{\frac{m}{3} - 2} + u^3\pi^{m-2} \equiv 0\pmod{\pi^{m-1}}\quad\text{if $m\equiv 2\pmod{4}$}.
    \end{cases}
    $$
    The result then follows from the fact that\footnote{This follows from expanding the binomial on the right-hand side of $$
    (\pi^4)^k = ((-a_0) + (-a_1\pi - a_2\pi^2 - a_3\pi^3))^k.
    $$
    }, for any positive integer $k$, we have 
    $$
    \pi^{4k} \equiv (-a_0)^k\pmod{\pi^{4k+\frac{2m}{3} - 2}}.
    $$
\end{proof}
\begin{lemma}
    \label{lem-T-m-iff-b0u-val}
    Let $4 \leq m \leq 6e_K + 2$ be a multiple of $6$ and let $f \in T_m$. Then $f \not \in P_m^{\oneaut}$ if and only if $v_K(b_0^{(u)}) \geq \frac{m}{3} + 1$ for some $u \in \mathcal{R}$. 
\end{lemma}
\begin{proof}
    Suppose that $f \not \in P_m^{\oneaut}$. Then $f$ has at least two roots in $L_f$. Reordering the $\sigma_i$ if necessary, we may assume that $\sigma_2(\pi) \in L_f$. Since $f \in T_m$, Lemma~\ref{lem-T-m-iff-val-m-by-12} tells us that $v_K(\sigma_2(\pi) - \pi) = \frac{m}{12}$, so 
    $$
    \sigma_2(\pi) = \pi + \tilde{u}\pi^{\frac{m}{3}}
    $$
    for some $\tilde{u} \in \co_{L_f}^\times$. Since $L_f/K$ is totally ramified, there is some $u\in \mathcal{R}$ with $u \equiv \tilde{u}\pmod{\pi}$, which means that 
    $$
    v_K(\sigma_2(\pi) - \pi - u\pi^{\frac{m}{3}}) > \frac{m}{12}.
    $$ 
    The other three roots of $g^{(u)}$ all have valuation at least $\frac{m}{12}$, so 
    $$
    v_K(b_0^{(u)}) \geq \frac{m}{3} + 1.
    $$
    Suppose conversely that $v_K(b_0^{(u)}) \geq \frac{m}{3} + 1$ for some $u \in \mathcal{R}$. Lemma~\ref{lem-vals-of-bi} tells us that $v_K(b_1^{(u)}) = \frac{m}{4}$ and $v_K(b_2^{(u)}) \geq \frac{m}{6}$, so considering the Newton polygon of $g^{(u)}$ tells us that it has exactly one root $\sigma_i(\pi) - \pi - u\pi^{\frac{m}{3}}$ with 
    $$
    v_\pi(\sigma_i(\pi) - \pi - u\pi^{\frac{m}{3}}) \geq \frac{m}{3} + 1.
    $$
    Therefore we have 
    $$
    \sigma_i(\pi) - \pi - u\pi^{\frac{m}{3}} \in L_f,
    $$
    so $\sigma_i(\pi) \in L_f$, which means that $f \not \in P^{\oneaut}$.
\end{proof}

\begin{corollary}
    \label{cor-cong-cond-m-multiple-of-6}
    Let $m$ be a multiple of $6$ with $4 \leq m \leq 6e_K + 2$, and let $f \in T_m$. The following are equivalent:
    \begin{enumerate}
        \item We have $f \not \in P_m^{\oneaut}$.
        \item There is some $u \in \mathcal{R}$ such that 
        $$
        \begin{cases}
            a_1 + ua_2a_0^{\lfloor\frac{m}{12}\rfloor} + u^3a_0^{\lfloor \frac{m}{4}\rfloor} \equiv 0 \pmod{\p_K^{\lfloor \frac{m}{4}\rfloor + 1}}\quad\text{if $m\equiv 0 \pmod{4}$},
            \\
            a_3 + ua_2a_0^{\lfloor\frac{m}{12}\rfloor} + u^3a_0^{\lfloor \frac{m}{4}\rfloor} \equiv 0 \pmod{\p_K^{\lfloor \frac{m}{4}\rfloor + 1}}\quad \text{if $m\equiv 2\pmod{4}$}.
        \end{cases}
        $$
    \end{enumerate}
\end{corollary}
\begin{proof}
    This is immediate from Lemmas~\ref{lem-vals-of-bi} and \ref{lem-T-m-iff-b0u-val}.
\end{proof}
\subsection{Computing the densities}
\begin{lemma}
    \label{lem-mu-Tm}
    Let $m$ be an even integer with $4 \leq m \leq 6e_K + 2$. Then 
    $$
    \mu(T_m) = q^{-\lceil \frac{2m}{3}\rceil - 3}(q-1)^2. 
    $$
\end{lemma}
\begin{proof}
    This is easy to see from the definition of $T_m$. 
\end{proof}
Since $\F_K \cong \F_{2^{f_K}}$, the trace map $\Tr_{\F_K/\F_2}: \F_K \to \F_2$ is given by 
$$
\Tr_{\F_K/\F_2}(x) = x + x^2 + \ldots + x^{2^{f_K-1}}. 
$$
\begin{lemma}
    \label{lem-num-roots-quadratic}
    Let $\alpha,\beta,\gamma\in \F_K$ with $\alpha \neq 0$, and let $g$ be the polynomial $\alpha X^2 + \beta X + \gamma$ in $\F_K[X]$. The number of roots of $g$ in $\F_K$ is 
    $$
    \begin{cases}
        1 \quad \text{if $\beta = 0$},
        \\
        2 \quad \text{if $\beta \neq 0$ and $\Tr_{\F_K/\F_2}(\alpha\gamma/\beta^2) = 0$},
        \\
        0 \quad \text{if $\beta\neq 0$ and $\Tr_{\F_K/\F_2}(\alpha\gamma/\beta^2) = 1$}.
    \end{cases}
    $$
\end{lemma}
\begin{proof}
    The case with $\beta = 0$ is clear, so assume $\beta \neq 0$. Let $u$ be a root of $g$ in a splitting field over $\F_K$, and let $\theta = \frac{\alpha u}{\beta}$. Clearly $u \in \F_K$ if and only if $\theta \in \F_K$, which is equivalent to $\theta + \theta^{q} = 0$. Since
    $$
    \Gal(\F_K/\F_2) = \{x\mapsto x^{2^i} : i = 0,1,\ldots, f_K - 1\},
    $$
    it is easy to see that 
    $$
    \Tr_{\F_K/\F_2}(\theta + \theta^2) = \theta + \theta^q,
    $$
    and also that 
    $$
    \theta + \theta^2 = \frac{\alpha\gamma}{\beta^2}.
    $$
    Therefore, $u \in \F_K$ if and only if $\Tr_{\F_K/\F_2}(\frac{\alpha\gamma}{\beta^2}) = 0$, and the result follows. 
\end{proof}
\begin{lemma}
    \label{lem-size-of-im-alpha}
    Let $n \geq 0$ be an integer and let $\lambda, \mu \in \p_K^n$, with $\mu \not \in \p_K^{n+1}$. Define the map 
    $$
    \alpha: \co_K/\p_K \to \co_K/\p_K^{n+1},\quad c\mapsto \lambda c + \mu c^3.
    $$
    The following two statements are true:
    \begin{enumerate}
    \item For $c \in (\co_K/\p_K)^\times$, we have 
    $$
    \#\{c'\in (\co_K/\p_K)^\times : \alpha(c') = \alpha(c)\} = \begin{cases}
        1 \quad \text{if $c^2 \equiv \lambda / \mu \pmod{\p_K}$},
        \\
        1 \quad \text{if $c^2 \neq \lambda/\mu$ and $\Tr_{\F_K/\F_2}\Big(\frac{\lambda}{c^2\mu}\Big) \not \equiv f_K \pmod{2}$},
        \\
        3 \quad \text{if $c^2 \neq \lambda/\mu$ and $\Tr_{\F_K/\F_2}\Big(\frac{\lambda}{c^2\mu}\Big) \equiv f_K \pmod{2}$}.
    \end{cases}
    $$
    \item We have 
    $$
    \#\im \alpha = \begin{cases}
        \frac{2q + (-1)^{f_K}}{3}\quad\text{if $\lambda \not \in \p_K^{n+1}$},
        \\
        \frac{q+1 + (-1)^{f_K}}{2 + (-1)^{f_K}}\quad\text{if $\lambda \in \p_K^{n+1}$}. 
    \end{cases}
    $$
\end{enumerate}
\end{lemma}
\begin{proof}
    It is easy to see that for $c,c' \in (\co_K/\p_K)^\times$, we have $\alpha(c) = \alpha(c')$ if and only if
    $$
    (c - c')\Big(
        (c')^2 + cc' + \frac{\lambda}{\mu} + c^2\Big) \equiv 0\pmod{\p_K}.
    $$
    The first statement then follows from Lemma~\ref{lem-num-roots-quadratic}. For the second statement, suppose first that $\lambda \not \in \p_K^{n+1}$. Then there is some $c \in (\co_K/\p_K)^\times$ with $\alpha(c) = 0$, so 
    \begin{align*}
    \#\im \alpha &= \sum_{c \in (\co_K/\p_K)^\times} \frac{1}{\# \{c' \in (\co_K/\p_K)^\times : \alpha(c') = \alpha(c)\}}
    \\
    &= 1 + (q-2-a) + \frac{a}{3},
    \end{align*}
    where 
    $$
    a = \#\{c \in (\co_K/\p_K)^\times : c^2 \neq \frac{\lambda}{\mu} \text{ and } \Tr_{\F_K/\F_2}\Big(\frac{\lambda}{c^2\mu}\Big) \equiv f_K\pmod{2}\}.
    $$
    Since $\lambda \not \in \p_K^{n+1}$, the map 
    $$
    (\p_K/\co_K)^\times \to (\p_K/\co_K)^\times,\quad c \mapsto \frac{\lambda}{c^2\mu}
    $$
    is a bijection, so 
    \begin{align*}
    a &= \#\{u \in (\co_K/\p_K)^\times\setminus\{1\} : \Tr_{\F_K/\F_2}(u) \equiv f_K\pmod{2}\}
    \\
    &= \frac{1}{2}(q - 3 - (-1)^{f_K}),
    \end{align*}
    and the result follows. Now suppose that $\lambda \in \p_K^{n+1}$. Then $\alpha(c) = 0$ if and only if $c = 0$, so 
    $$
        \#\im\alpha = 1 + \sum_{c \in (\co_K/\p_K)^\times} \frac{1}{\# \{c' \in (\co_K/\p_K)^\times : \alpha(c') = \alpha(c)\}}.
    $$
    We have $\frac{\lambda}{c^2\mu} \equiv 0\pmod{\p_K}$ for all $c \in (\co_K/\p_K)^\times$, so 
    $$
    \# \{c' \in (\co_K/\p_K)^\times : \alpha(c') = \alpha(c)\} = 2 + (-1)^{f_K},
    $$
    and the result follows. 
\end{proof}
\begin{lemma}
    \label{lem-density-of-cubic-cong-cond}
    Let $a, b > 0$ be integers, and let $S$ be the set of triples $(x_0,x_1,x_2) \in \co_K^3$ such that the following two conditions hold: 
    \begin{enumerate}
    \item $
    v_K(x_0) = 1, \quad v_K(x_1) = a + b, \quad v_K(x_2) \geq b.
    $
    \item There is some $u \in \mathcal{R}$ such that $x_1 + ux_2x_0^{a} + u^3x_0^{a+b} \equiv 0 \pmod{\p_K^{a + b + 1}}$. 
    \end{enumerate}
    Then $\mu(S) = \frac{1}{3}q^{-a - 2b - 4}(q-1)^2(2q-1)$.
\end{lemma}
\begin{proof}
    Suppose that, for $x_i$ and $x_i'$ in $\co_K$, we have $x_i \equiv x_i'\pmod{\p_K^{a + b + 1}}$ for $i=0,1,2$. Then $(x_0,x_1,x_2)\in S$ if and only if $(x_0',x_1',x_2') \in S$, so
    $$
    \mu(S) = \frac{\#\overline{S}}{q^{3a + 3b + 3}},
    $$
    where $\overline{S}$ is the set of triples 
    $$
    (\bar{x}_0,\bar{x}_1,\bar{x}_2) \in 
    \Big((\p_K/\p_K^{a + b + 1})\setminus (\p_K^2/\p_K^{a + b + 1})\Big) 
    \times 
    \Big((\p_K^{a + b}/\p_K^{a+b+1})\setminus \{0\}\Big)
    \times 
    (\p_K^{b}/\p_K^{a + b + 1}) 
    $$
    such that there is some $u \in \mathcal{R}$ with 
    $$
    \bar x_1 + u\bar x_2\bar x_0^{a} + u^3 \bar x_0^{a+b} = 0.
    $$
    For each $\bar x_0 \in (\p_K/\p_K^{a +b+ 1})\setminus (\p_K^2/\p_K^{a +b+ 1})$ and $\bar x_2 \in \p_K^{b}/\p_K^{a +b+ 1}$, define the map 
    $$
    \alpha_{\bar{x}_0,\bar{x}_2}: \co_K/\p_K \to \p_K^{a+b}/\p_K^{a +b+ 1},\quad u \mapsto - u\bar x_2\bar x_0^{a} - u^3 \bar x_0^{a+b}.
    $$
    Then 
    $$
    \overline{S} = \bigsqcup_{\substack{\bar x_0 \in (\p_K/\p_K^{a+b + 1})\setminus (\p_K^2/\p_K^{a+b + 1})\\ \bar x_2 \in \p_K^{b}/\p_K^{a+b + 1}}} \{\bar{x}_0\}\times \Big(\im\alpha_{\bar x_0,\bar x_2} \setminus\{0\}\Big) \times \{\bar x_2\}.
    $$
    Since $\alpha_{\bar{x}_0,\bar{x}_2}(0) = 0$, we always have $0 \in \im  \alpha_{\bar{x}_0,\bar{x}_2}$, so 
    $$
    \#\Big(\im\alpha_{\bar x_0,\bar x_2} \setminus\{0\}\Big) = \#\im\alpha_{\bar x_0,\bar x_2} - 1,
    $$
    and therefore
    $$
    \# \overline{S} = \sum_{\substack{\bar x_0 \in (\p_K/\p_K^{a+b + 1})\setminus (\p_K^2/\p_K^{a +b+ 1})\\ \bar x_2 \in \p_K^{b}/\p_K^{a+b + 1}}} (\#\im\alpha_{\bar x_0,\bar x_2} - 1),
    $$
    Lemma~\ref{lem-size-of-im-alpha} tells us that 
    $$
    \#\im\alpha_{\bar x_0, \bar x_2} = \begin{cases}
        \frac{2q + (-1)^{f_K}}{3} \quad\text{if $\bar{x}_2 \not \in \p_K^{b + 1}/\p_K^{a+b+ 1}$},
        \\
        \frac{q+1+(-1)^{f_K}}{2 + (-1)^{f_K}} \quad \text{if $\bar{x}_2  \in \p_K^{b + 1}/\p_K^{a+b + 1}$}. 
    \end{cases}
    $$
    It follows that 
    $$
    \# \overline{S} = \frac{1}{3}q^{2a + b - 1}(q-1)^2(2q-1),
    $$
    so 
    $$
    \mu(S) = \frac{1}{3}q^{-a - 2b - 4}(q-1)^2(2q-1). 
    $$
\end{proof}
\begin{corollary}
    \label{cor-Tm-minus-Pm}
    Let $4 \leq m \leq 6e_K + 2$ be a multiple of $6$. Then 
    $$
    \mu(T_m \setminus P_m^{\oneaut}) = \frac{1}{3}q^{-\frac{2m}{3} - 4}(q-1)^2(2q-1).
    $$
\end{corollary}
\begin{proof}    
    Suppose first that $4 \mid m$. Setting $x_i = a_i$ for $i=0,1,2$ and $(a,b) = (\frac{m}{12}, \frac{m}{6})$, Corollary~\ref{cor-cong-cond-m-multiple-of-6} tells us that $T_m\setminus P_m^{\oneaut}$ is the set $S$ from Lemma~\ref{lem-density-of-cubic-cong-cond}, together with the added congruence condition that $v_K(a_3) \geq \frac{m}{4}$, so 
    $$
    \mu(T_m\setminus P_m^{\oneaut}) = \mu(S)\cdot q^{- \frac{m}{4}} = \frac{1}{3}q^{-\frac{2m}{3} - 4}(q-1)^2(2q-1).
    $$ 
    If $4 \nmid m$, then set 
    $$
    (x_0,x_1,x_2) := (a_0, a_3, a_2),\quad (a,b) = \Big(\frac{m-6}{12}, \frac{m}{6}\Big),
    $$
    and proceed similarly. 
\end{proof}
\begin{corollary}
    \label{cor-mu-Pm}
    Let $4 \leq m \leq 6e_K + 2$ be an even integer. Then 
    $$
    \mu(P_m^{\oneaut}) = q^{ - \lceil \frac{2m}{3}\rceil - 3}(q-1)^2 \cdot \Big(1 + \mathbbm{1}_{6 \mid m}\cdot \Big(\frac{1 - 2q}{3q}\Big)\Big).
    $$
\end{corollary}
\begin{proof}
    This is immediate from Corollary~\ref{cor-Tm-minus-Pm} and Lemma~\ref{lem-mu-Tm}.
\end{proof}
\begin{corollary}
    \label{cor-size-of-Sigma-1-m}
    If $\Sigma_m^{\oneaut}$ is nonempty, then $m$ is an even integer with $4 \leq m \leq 6e_K + 2$, and 
    $$
    \#\Sigma_m^{\oneaut} = q^{\lfloor\frac{m}{3}\rfloor - 1}(q-1)\Big(1 + \mathbbm{1}_{6 \mid m}\cdot \Big(\frac{1 - 2q}{3q}\Big)\Big).
    $$
\end{corollary}
\begin{proof}
    This is immediate from Lemma~\ref{lem-num-exts-density-polys} and Corollary~\ref{cor-mu-Pm}.
\end{proof}

\subsection{Distinguishing between $A_4$ and $S_4$}

Write ``$\mu_3\subseteq K$'' as shorthand for ``$K$ contains three distinct cube roots of unity''.
\begin{lemma}
    \label{lem-A4-vs-S4-basic-facts}
    The following three statements are true: 
    \begin{enumerate}
        \item (Tower law for discriminant) Let $M/L/K$ be extensions of $2$-adic fields. Then 
        $$
        v_K(d_{M/K}) = [M : L] \cdot v_K(d_{L/K}) + f(L/K)\cdot v_L(d_{M/L}). 
        $$
        \item We have $\mu_3\subseteq K$ if and only if $f_K$ is even. 
        \item If $\mu_3\not\subseteq K$, then $K$ has only one $C_3$-extension up to isomorphism, namely the unramified extension. 
    \end{enumerate}
\end{lemma}
\begin{proof}
    Claim (1) is \cite[Proposition III.8]{greenberg1995local}. Claim (2) follows from Hensel's Lemma. Finally, Claim (3) comes from class field theory, along with the fact that $K^\times / K^{\times 3} \cong \Z/3\Z$, which follows from \cite[Proposition~3.7]{neukirch-bonn}. 
\end{proof}
\begin{lemma}
    \label{lem-mu3-no-S4}
    If $\mu_3\subseteq K$, then $K$ has no $S_4$-extensions. 
\end{lemma}
\begin{proof}
    This is part of \cite[Theorem~1.2]{wei-ji}.
\end{proof}
\begin{proof}[Proof of Theorem~\ref{thm-S4-A4-f-even}]
    By Lemma~\ref{lem-A4-vs-S4-basic-facts} (2), we have $\mu_3 \subseteq K$, so Lemma~\ref{lem-mu3-no-S4} tells us that $\Sigma^{A_4}_m = \Sigma^{\oneaut}_m$, and the result follows by Corollary~\ref{cor-size-of-Sigma-1-m}. 
\end{proof}

\begin{lemma}
    \label{lem-disc-of-V4-in-terms-of-quad-exts}
    Let $M/F$ be a $V_4$-extension of $2$-adic fields, and let $E_1,E_2,E_3$ be its three quadratic intermediate extensions. Then 
    $$
    v_F(d_{M/F}) = \sum_{i=1}^3 v_F(d_{E_i/F}). 
    $$
\end{lemma}
\begin{proof}
    This follows easily from \cite[Theorem~17.50]{keune-nf}. 
\end{proof}

\begin{lemma}
    \label{lem-A4-quartic-implies-disc-mult-3}
    Suppose that $\mu_3\not\subseteq K$ and let $L \in \Sigma^{A_4}$. Then $3 \mid v_K(d_{L/K})$. 
\end{lemma}
\begin{proof}
    Let $M$ be a normal closure of $L$ over $K$, so $\Gal(M/K) \cong A_4$, and let $F = M^{V_4}$. The extension $F/K$ is a $C_3$-extension, so it is unramified by Lemma~\ref{lem-A4-vs-S4-basic-facts}, Part (3). Since $L/K$ is totally ramified, we have $e(M/K) = 4$ and $f(M/K) = 3$, so $V_4$ is the inertia group of $M/K$. Since $F/K$ is unramified, the tower law for discriminant gives 
    $$
    v_K(d_{M/K}) = 3v_F(d_{M/F}). 
    $$
    Let $E_1,E_2, E_3$ be the three intermediate extensions of the $V_4$-extension $M/F$. Since the three double transpositions in $A_4$ are conjugate, the extensions $E_i/K$ are isomorphic, so they have the same discriminant. By the tower law for discriminant, it follows that the valuations
    $$
    v_{F}(d_{E_i/F}),\quad i = 1,2,3
    $$
    are all equal. By Lemma~\ref{lem-disc-of-V4-in-terms-of-quad-exts}, we have 
    $$
    v_F(d_{M/F}) = \sum_{i=1}^3 v_F(d_{E_i/F}) = 3v_F(d_{E_1/F}), 
    $$
    so 
    $$
    v_K(d_{M/K}) = 9v_F(d_{E_1/F}). 
    $$
    Since $M/L$ is unramified, the tower law also gives 
    $$
    v_K(d_{M/K}) = 3v_K(d_{L/K}),
    $$
    and the result follows.
\end{proof}

\begin{lemma}
    \label{lem-A4-exts-tot-ram}
    Suppose that $\mu_3 \not \subseteq K$. Then there is a bijection between $\Sigma^{A_4}$ and the set of isomorphism classes of $A_4$-extensions of $K$. 
\end{lemma}
\begin{proof}
    For an $A_4$-quartic extension $L/K$, let $\widetilde{L}$ be the normal closure of $L$ over $K$. The map $L \mapsto \widetilde{L}$ is a well-defined bijection between the set of isomorphism classes of $A_4$-quartics and the set of isomorphism classes of $A_4$-extensions. To prove the lemma, it suffices to show that every $A_4$-extension of $K$ is the normal closure of a totally ramified $A_4$-quartic. 

    Let $M/K$ be an $A_4$-extension with inertia group $G_0\subseteq A_4$. Then $M^{V_4}/K$ is a $C_3$-extension, so it is unramified by Lemma~\ref{lem-A4-vs-S4-basic-facts}, Part (3), and therefore $G_0 \subseteq V_4$. Since $M/K$ is not cyclic, it is ramified, so $\# G_0 \geq 2$. Since $G_0$ is a normal subgroup of $A_4$, we must have $G_0 = V_4$, so $e(M/K) = 4$. Therefore, for any choice of $A_3\subseteq A_4$, the field $L := M^{A_3}$ is a totally ramified $A_4$-quartic extension of $K$ with normal closure $M$, so we are done. 
\end{proof}
\begin{lemma}
    \label{lem-sigma-A4-in-terms-of-Sigma-1}
    Suppose that $\mu_3\not\subseteq K$. We have 
    $$
    \Sigma^{A_4} = \bigcup_{\substack{m \\ 6 \mid m}} \Sigma_m^{\oneaut}. 
    $$
\end{lemma}
\begin{proof}
    By Corollary~\ref{cor-size-of-Sigma-1-m} and Lemma~\ref{lem-A4-quartic-implies-disc-mult-3}, we have 
    $$
    \Sigma^{A_4} \subseteq \bigcup_{\substack{m \\ 6 \mid m}} \Sigma_m^{\oneaut}. 
    $$
    Lemma~\ref{lem-A4-exts-tot-ram} and \cite[Theorem~1.2]{wei-ji} tell us that 
    $$
    \#\Sigma^{A_4} = \frac{q^{2e_K} - 1}{3}. 
    $$
    From Corollary~\ref{cor-size-of-Sigma-1-m}, we obtain
    $$
    \sum_{\substack{m \\ 6\mid m}} \#\Sigma_m^{\oneaut} = \frac{q^{2e_K} - 1}{3},
    $$
    and the result follows. 
\end{proof}
\begin{proof}[Proof of Theorem~\ref{thm-S4-A4-f-odd}]
    Lemma~\ref{lem-A4-vs-S4-basic-facts}(2) tells us that $\mu_3 \not\subseteq K$. The result then follows from Corollary~\ref{cor-size-of-Sigma-1-m} and Lemma~\ref{lem-sigma-A4-in-terms-of-Sigma-1}.
\end{proof}

\begin{proof}[Proof of Corollary~\ref{cor-premass-of-S4-A4-f-even}]
    Theorem~\ref{thm-S4-A4-f-even} tells us that $\m(\Sigma^{S_4}) = 0$ and
    $$
    \m(\Sigma^{A_4}) = \sum_{\substack{4 \leq m \leq 6e_K + 2 \\ \text{$m$ even}}} q^{- \lceil \frac{2m}{3}\rceil - 1}(q-1)\Big(1 + \mathbbm{1}_{6 \mid m}\cdot \big(\frac{1 - 2q}{3q}\big)\Big).
    $$
    Setting $m = 2k$, this is the same as 
    $$
    \sum_{k=2}^{3e_K + 1}q^{- \lceil \frac{4k}{3}\rceil - 1}(q-1)\Big(1 + \mathbbm{1}_{3\mid k}\big(\frac{1-2q}{3q}\big)\Big),
    $$
    which is equal to 
    $$
    (q-1)\sum_{l=1}^{e_K}\Big(
        q^{-4l} + \frac{1}{3}q^{-4l-2}(q+1) + q^{-4l-3}
    \Big) = (q-1)\sum_{l=1}^{e_K}\Big(
        q^{-4l-3} + \frac{1}{3}q^{-4l-2} + \frac{1}{3}q^{-4l-1} + q^{-4l}    
    \Big).
    $$
    It is easy to see that 
    $$
    \sum_{l=1}^{e_K} q^{-4l} = q^{-4e_K}\cdot \frac{q^{4e_K}-1}{q^4-1},
    $$
    so the quantity we are computing equals 
    $$
    (q-1)\cdot\frac{q^{4e_K}-1}{q^4-1}\cdot \Big(
       q^{-4e_K-3} + \frac{1}{3}q^{-4e_K-2} + \frac{1}{3}q^{-4e_K-1} + q^{-4e_K}
    \Big),
    $$
    and this is equal to 
    $$
    \frac{1}{3}(q-1)\cdot\frac{q^{4e_K}-1}{q^4-1}\cdot q^{-4e_K-3}\Big(
        3q^3 + q^2 + q + 3
    \Big).
    $$
\end{proof}

\begin{proof}[Proof of Corollary~\ref{cor-premass-of-S4-A4-f-odd}]
    By Theorem~\ref{thm-S4-A4-f-odd}, we have 
    \begin{align*}
        \m(\Sigma^{S_4}) &= \sum_{\substack{4 \leq m \leq 6e_K + 2 \\ 2 \mid m, \hspace{0.25em} 3\nmid m}} q^{\lfloor\frac{m}{3} \rfloor - m - 1}(q-1)
        \\
        &= q^{-1}(q-1)\sum_{\substack{2 \leq k \leq 3e_K + 1 \\ 3\nmid k}}q^{-\lceil \frac{4k}{3}\rceil}
        \\
        &= q^{-1}(q-1)\sum_{l=1}^{e_K}(q^{-2 - 4l} - q^{1 - 4l})
        \\
        &= q^{-1}(q-1)(q + q^{-2})\sum_{l=1}^{e_K}q^{-4l}
        \\
        &= q^{-3}(q-1)(q^3+1)\cdot \frac{1 - q^{-4e_K}}{q^4 - 1}
        \\
        &= \frac{q^3 + 1}{q^3 + q^2 + q + 1}\cdot (q^{-3} - q^{-4e_K - 3}). 
    \end{align*}
    The computation for $\m(\Sigma^{A_4})$ is similar but easier, so we omit it. 
\end{proof}

\section{The Case $G = V_4$}
\label{sec-G=V4}

\begin{lemma}
    \label{lem-half-of-hecke}
    Let $d \in K^\times\setminus K^{\times 2}$ and let $E = K(\sqrt{d})$. If $v_K(d)$ is even, then $v_K(d_{E/K})$ is an even integer with $2 \leq v_K(d_{E/K})\leq 2e_K$. If $v_K(d)$ is odd, then $v_K(d_{E/K}) = 2e_K + 1$.
\end{lemma}
\begin{proof}
    This is part of the $p=2$ case of \cite[Theorem~2.4]{hecke-theorem}.
\end{proof}

\begin{lemma}
    \label{lem-V4-ext-implies-even-m}
    Suppose that $\Sigma_m^{V_4}$ is nonempty. Then $m$ is an even integer and $6 \leq m \leq 6e_K + 2$. 
\end{lemma}
\begin{proof}
    Let $L \in \Sigma_m^{V_4}$, and let $E_1, E_2$ and $E_3$ be the intermediate quadratic subfields of $L$. Let $c_i = v_K(d_{E_i/K})$ for each $i$, so that
    $$
    m = c_1 + c_2 + c_3,
    $$
    by Lemma~\ref{lem-disc-of-V4-in-terms-of-quad-exts}.
    We may write $E_i = K(\sqrt{d_i})$, for $d_i \in K^\times\setminus K^{\times 2}$, such that $d_1d_2d_3 \in K^{\times 2}$.  Since $v_K(d_1d_2d_3)$ is even, it follows from Lemma~\ref{lem-half-of-hecke} that either $0$ or $2$ of the $c_i$ are equal to $2e_K + 1$, and the rest are even integers with $2 \leq c_i \leq 2e_K$. The result follows. 
\end{proof}

\begin{lemma}
    \label{lem-size-Sigma_V4-m}
{\cite[Lemma~4.7]{tunnell}} Let $m$ be a positive even integer with $2 \leq m \leq 6e_K + 2$. Then 
$$
\#\Sigma_m^{V_4} = 2(q-1)q^{\frac{m-4}{2}}\Big(
    q^{-\lfloor \frac{m}{6}\rfloor}(1 + \mathbbm{1}_{3 \mid m}\cdot \frac{q-2}{3}) - \mathbbm{1}_{m\leq 4e_K + 2}\cdot q^{-\lfloor\frac{m-2}{4}\rfloor}
\Big)
$$
\end{lemma}
\begin{remark}
    Tunnell omits the statement that he only counts totally ramified extensions. This fact can be seen from the second paragraph of his proof, where he writes $d > \frac{c-1}{3}$. The fact that this inequality is strict means that all intermediate quadratic fields have positive discriminant, so the $V_4$-extension is totally ramified. 
\end{remark}
\begin{proof}[Proof of Theorem~\ref{thm-size-of-Sigma-V4-m}]
    The result follows immediately from Lemmas~\ref{lem-V4-ext-implies-even-m} and \ref{lem-size-Sigma_V4-m}.
\end{proof}

\begin{proof}[Proof of Corollary~\ref{cor-premass-C2xC2}]
    By Lemmas~\ref{lem-V4-ext-implies-even-m} and \ref{lem-size-Sigma_V4-m}, we have 
    \begin{align*}
    \m(\Sigma^{V_4}) &= \sum_{\substack{4 \leq m \leq 6e_K + 2 \\ \text{$m$ even}}} \frac{\#\Sigma_m^{V_4}}{4q^m}
    \\
    &= \frac{1}{2}(q-1)\cdot\Big(\sum_{\substack{4 \leq m \leq 6e_K + 2 \\ \text{$m$ even}}} q^{-\frac{m+4}{2} - \lfloor \frac{m}{6}\rfloor}(1 + \mathbbm{1}_{3\mid m}\cdot \frac{q-2}{3}) - \sum_{\substack{4 \leq m \leq 4e_K + 2 \\ \text{$m$ even}}} q^{- \frac{m+4}{2} - \lfloor \frac{m-2}{4}\rfloor}\Big).
    \end{align*}
    We have 
    \begin{align*}
        \sum_{\substack{4 \leq m \leq 6e_K + 2 \\ \text{$m$ even}}} q^{-\frac{m+4}{2} - \lfloor \frac{m}{6}\rfloor}(1 + \mathbbm{1}_{3\mid m}\cdot \frac{q-2}{3}) &= \sum_{k=2}^{3e_K + 1} q^{-k - 2 - \lfloor \frac{k}{3}\rfloor}(1 + \mathbbm{1}_{3\mid k}\cdot \frac{q-2}{3})
        \\
        &= \sum_{l=1}^{e_K}\Big(q^{-4l} + q^{-4l - 2}\cdot \frac{q+1}{3} + q^{-4l-3}\Big).
    \end{align*}
    It is easy to see that 
    $$
    \sum_{l=1}^{e_K}q^{-4l} = q^{-4e_K}\cdot \frac{q^{4e_K} - 1}{q^4 - 1},
    $$
    which means that 
    $$
        \sum_{\substack{4 \leq m \leq 6e_K + 2 \\ \text{$m$ even}}} q^{-\frac{m+4}{2} - \lfloor \frac{m}{6}\rfloor}(1 + \mathbbm{1}_{3\mid m}\cdot \frac{q-2}{3}) = \frac{1}{3}\cdot q^{-4e_K-3}\cdot \frac{q^{4e_K} - 1}{q^4 - 1} \cdot (3q^3 + q^2 + q + 3)
    $$
    Similarly, we obtain 
    $$
    \sum_{\substack{4 \leq m \leq 4e_K + 2 \\ \text{$m$ even}}} q^{- \frac{m+4}{2} - \lfloor \frac{m-2}{4}\rfloor} = q^{-3e_K - 3}\cdot \frac{q^{3e_K} - 1}{q^3 - 1} \cdot (q^2 + 1),
    $$
    and the result follows. 
\end{proof}

\section{The Case $G = C_4$}
\label{sec-G=C4}

\subsection{Sketch of our approach}

Let $L/K$ be a $C_4$-extension and let $E$ be its unique nontrivial intermediate field. 

For a $2$-adic field $F$, write $\Sigma_{\quadr/F}$ for the set of isomorphism classes of quadratic extensions of $F$. For any real number $m$, write $\Sigma_{\quadr/F,m}$ (respectively $\Sigma_{\quadr/F,\leq m}$) for the set of $E \in \Sigma_{\quadr/F}$ with $v_F(d_{E/F}) = m$ (respectively $v_F(d_{E/F}) \leq m$). For quadratic extensions $E/K$ and $G \in \{D_4, V_4,C_4\}$, write $\Sigma_{\quadr/E}^{G/K}$ for the set of $L \in \Sigma_{\quadr/E}$ such that $L/K$ has Galois closure group isomorphic to $G$. Finally, for any real number $m_2$, write $\Sigma_{\quadr/E,m_2}^{G/K}$  (respectively $\Sigma_{\quadr/E,\leq m_2}^{G/K}$) for $\Sigma_{\quadr/E}^{G/K} \cap \Sigma_{\quadr/E,m_2}$ (respectively $\Sigma_{\quadr/E}^{G/K} \cap \Sigma_{\quadr/E,\leq m_2}$).

Call a quadratic extension $E/K$ \emph{$C_4$-extendable} if there is some quadratic extension $L/E$ such that $L/K$ is a $C_4$-extension. For any real number $m_1$, write $\Sigma^{\extendable}_{\quadr/K,m_1}$ (respectively $\Sigma^{\extendable}_{\quadr/K,\leq m_1}$) for the set of $C_4$-extendable extensions $E/K$ such that $v_K(d_{E/K}) = m_1$ (respectively $v_K(d_{E/K}) \leq m_1$). 

Recall that we write $d_{(-1)} = v_K(d_{K(\sqrt{-1})/K})$. In the current subsection, we state the main results, whose proofs are postponed to the later subsections. 
\begin{definition}
    \label{defi-N-ext}
    For even integers $m_1$ with $2 \leq m_1 \leq 2e_K$, define 
$$
N_{\ext}(m_1) := (1 + \mathbbm{1}_{m_1\leq 2e_K - d_{(-1)}})q^{\frac{m_1}{2} - 1}(q - 1 - \mathbbm{1}_{m_1 = 2e_K - d_{(-1)} + 2}).
$$
For $m_1 = 2e_K + 1$, define
    $$
    N_{\ext}(2e_K + 1) = \begin{cases}
        2q^{e_K}\quad\text{if $-1 \in K^{\times 2}$},
        \\
        q^{e_K} \quad\text{if $K(\sqrt{-1})/K$ is quadratic and totally ramified},
        \\
        0\quad\text{if $K(\sqrt{-1})/K$ is quadratic and unramified}.
    \end{cases}
    $$
And set $N_{\ext}(m_1) = 0$ for all other real numbers $m_1$. 
\end{definition}
\begin{lemma}
    \label{lem-N-ext}
    If $E/K$ is a totally ramified $C_4$-extendable extension, then $2 \leq v_K(d_{E/K}) \leq 2e_K + 1$ and $v_K(d_{E/K})$ is either even or equal to $2e_K + 1$. For such $m_1$, we have 
    $$
    \#\Sigma^{\extendable}_{\quadr/K,m_1} = N_{\ext}(m_1). 
    $$
\end{lemma}

\begin{definition}
    \label{defi-N-C4}
    Let $m_1$ be an even integer with $2 \leq m_1 \leq e_K$. For each integer $m_2$, define 
    $$
    N^{C_4}(m_1,m_2) = \begin{cases}
        q^{m_1-1}\quad\text{if $m_2 = 3m_1 - 2$},
        \\
        q^{\lfloor \frac{m_1+m_2}{4}\rfloor} - q^{\lfloor \frac{m_1+m_2-2}{4}\rfloor} \quad\text{if $3m_1 \leq m_2 \leq 4e_K - m_1$ and $m_2$ is even},
        \\
        q^{e_K} \quad\text{if $m_2 = 4e_K - m_1 + 2$},
        \\
        0 \quad\text{otherwise}. 
    \end{cases}
    $$
    Suppose that $m_1 = 2e_K + 1$ or $m_1$ is even with $e_K < m_1 \leq 2e_K$. Then define 
    $$
    N^{C_4}(m_1,m_2) = \begin{cases}
        2q^{e_K}\quad\text{if $m_2 = m_1 + 2e_K$},
        \\
        0\quad\text{otherwise}. 
    \end{cases}
    $$
    Finally, define $N^{C_4}(m_1,m_2) = 0$ for all other real numbers $m_1$ and $m_2$. 
\end{definition}
\begin{lemma}
    \label{lem-N-E-m2}
    Let $E$ be a totally ramified $C_4$-extendable extension and let $m_1 = v_K(d_{E/K})$. For all $m_2$, we have 
    $$
    \#\Sigma_{\quadr/E,m_2}^{C_4/K} = N^{C_4}(m_1,m_2). 
    $$
\end{lemma}

\begin{corollary}
    \label{cor-size-of-Sigma-m-C4}
    If $\Sigma_{m}^{C_4}$ is nonempty, then either $m=8e_K + 3$ or $m$ is an even integer with $8 \leq m \leq 8e_K$. For any even integer $m$, the number $\# \Sigma^{C_4}_m$ is the sum of the following four quantities:
    \begin{enumerate}
        \item $\mathbbm{1}_{8\leq m \leq 5e_K - 2}\cdot q^{\frac{m-3}{5}}N_{\ext}(\frac{m+2}{5})$.
        \item $$\sum_{\substack{\max\{2, m-4e_K\}\leq m_1 \leq \min\{\frac{m}{5},e_K\} \\ m_1\equiv m\pmod{4}}} q^{\frac{m-m_1}{4} - 1}(q-1)N_{\ext}(m_1).$$ 
        \item $\mathbbm{1}_{4e_K + 4 \leq m \leq 5e_K + 2}\cdot q^{e_K} N_{\ext}(m - 4e_K - 2).$
        \item $
        \mathbbm{1}_{5e_K + 3 \leq m \leq 8e_K} \cdot 2q^{e_K}N_{\ext}(\frac{m-2e_K}{3}). 
        $
    \end{enumerate}
    Moreover, 
    $$
    \#\Sigma_{8e_K + 3}^{C_4} = \begin{cases}
        4q^{2e_K}\quad\text{if $-1 \in K^{\times 2}$},
        \\
        2q^{2e_K} \quad\text{if $K(\sqrt{-1})/K$ is quadratic and totally ramified},
        \\
        0\quad\text{if $K(\sqrt{-1})/K$ is quadratic and unramified}.
    \end{cases}
    $$
\end{corollary}

\subsection{Counting $C_4$-extendable extensions}
The aim of this subsection is to prove Lemma~\ref{lem-N-ext}. The paper \cite{cohen-et-al} gives conditions on $d \in K^\times$ for the extension $K(\sqrt{d})/K$ to be $C_4$-extendable. We use these conditions and adapt the methods of \cite{cohen-et-al} to parametrise and count $C_4$-extendable extensions. 
\begin{lemma}[Hecke's Theorem]
    \label{lem-hecke}
    Let $E$ be a $2$-adic field, let $\alpha \in E^\times \setminus E^{\times 2}$, and let $L = E(\sqrt{\alpha})$. If $v_E(\alpha)$ is odd, then $v_E(d_{L/E}) = 2v_E(2) + 1$. If $v_E(\alpha)$ is even, then $L/E$ is totally ramified if and only if $\alpha / x^2\equiv 1 \pmod{\p_E^{2v_E(2)}}$ has no solution $x \in E$. In that case, we have
    $$
    v_E(d_{L/E}) = 2v_E(2) + 1 - \kappa_{E,\alpha},
    $$
    where 
    $$
    \kappa_{E,\alpha} = \max\{0 \leq l < 2v_E(2) : \alpha / x^2\equiv 1 \pmod{\p_E^l} \text{ has a solution in $E$}\}.
    $$
\end{lemma}
\begin{proof}
    This is the special case $p=2$ of \cite[Theorem~2.4]{hecke-theorem}.
\end{proof}

\begin{corollary}
    \label{cor-disc-square-mod-2t+1}
    Let $E, \alpha$, and $L$ be as in Lemma~\ref{lem-hecke}, and assume that $v_E(\alpha)$ is even. Let $t$ be an integer with $0 \leq t \leq v_E(2)$. Then $v_E(d_{L/E})$ is an even integer and 
    $$
    v_E(d_{L/E}) \leq 2v_E(2) - 2t
    $$ 
    if and only if there is some $x \in E^\times$ with $\alpha / x^2\equiv 1 \pmod{\p_E^{2t}}$. 
\end{corollary}
\begin{proof}
    This follows from Lemma~\ref{lem-hecke}, along with the fact\footnote{If $u \equiv x^2\pmod{\p_E^{2t}}$, then $u/x^2 = 1 + \pi_E^{2t}y$ for some $y \in \co_E$. Taking $z \in \co_E$ with $y \equiv z^2\pmod{\p_E}$, we obtain $u/x^2 \equiv  (1 + \pi_E^tz)^2\pmod{\p_E^{2t+1}}$.} that for $0 \leq t < v_E(2)$ and $u \in \co_E^\times$, if $u$ is square modulo $\p_E^{2t}$ then it is also square modulo $\p_E^{2t+1}$.
\end{proof}
\begin{lemma}
    \label{lem-GG-norm-conditions}
    Let $E = K(\sqrt{d})$ for $d \in K^{\times}\setminus K^{\times 2}$ and let $L = E(\sqrt{\alpha})$ for $\alpha \in E^\times\setminus E^{\times 2}$. The Galois closure group of $L/K$ is 
    $$
    \begin{cases}
        V_4\quad\text{if $N_{E/K}(\alpha)\in K^{\times 2}$},
        \\
        C_4 \quad\text{if $N_{E/K}(\alpha) \in dK^{\times 2}$},
        \\
        D_4\quad\text{otherwise.}
    \end{cases}
    $$
\end{lemma}

\begin{proof}
    Write $\alpha = a + b\sqrt{d}$ for $a,b \in K$ and let $\theta = \sqrt{\alpha}$. Let $m(X)$ be the minimal polynomial of $\theta$ over $K$. Let $N$ be a splitting field of $m(X)$ over $L$. The polynomial $m(X)$ has roots $\pm \theta, \pm \varphi$ for some element $\varphi \in N$. 

    We claim that $L/K$ is a $V_4$-extension if and only if $\theta\varphi \in K$. Suppose that $L/K$ is a $V_4$-extension. Since $L/K$ is the splitting field of $m(X)$, there are $\sigma,\tau \in \Gal(L/K)$ with $\sigma(\theta) = \varphi$ and $\tau(\theta) = - \theta$. These have order $2$, so $\sigma(\theta\varphi) = \tau(\theta\varphi) = \theta\varphi$, and therefore $\theta\varphi \in K$. Suppose conversely that $\theta\varphi \in K$. Then $K(\theta) = K(\varphi)$, so $L$ is the splitting field of $m(X)$ over $K$, and therefore there are $\sigma,\tau \in \Gal(L/K)$ with $\sigma(\theta) = \varphi$ and $\tau(\theta) = - \theta$. Since $\theta\varphi \in K$, it is fixed by $\sigma$, so 
    $$
    \theta\varphi = \varphi\sigma(\varphi) ,
    $$
    and therefore $\theta = \sigma(\varphi)$, so $\sigma$ has order $2$. Clearly $\tau$ has order $2$, so $\Gal(L/K) \cong V_4$. 
    
    Let $\lambda := \frac{\theta}{\varphi} - \frac{\varphi}{\theta}$. We claim that $L/K$ is a $C_4$-extension if and only if $\lambda \in K$. Suppose that $L/K$ is a $C_4$-extension. Then $\theta, \varphi \in L$, so there is a generator $\sigma \in \Gal(L/K)$ such that $\sigma(\theta) = \varphi$. It follows that $\sigma(\lambda) = \lambda$, so $\lambda \in K$. Suppose conversely that $\lambda \in K$. There is some element $\sigma \in \Gal(N/K)$ such that $\sigma(\theta) = \varphi$. It is easy to see that $\sigma^2(\theta) = \varepsilon \theta$ for some $\varepsilon \in \{\pm 1\}$. Since $\lambda\in K$, we have $\varepsilon = - 1$, so $\sigma$ has order $4$. Clearly $\theta^2 + \varphi^2 = 2a$, so 
    $$
    \lambda = \frac{2\theta^2 - 2a}{\theta \varphi},
    $$
    which means that 
    $$
    \varphi = \frac{2\theta^2 - 2a}{\theta \lambda} \in L,
    $$
    so $L/K$ is Galois and hence $C_4$ with Galois group $\langle \sigma \rangle$. Finally, 
    $$
    \lambda^2 = \frac{4b^2d}{N_{E/K}(\alpha)},
    $$
    and the result follows. 
\end{proof}
\begin{corollary}
    \label{cor-extendable-iff-so2s}
    For $d \in K^\times \setminus K^{\times 2}$, the following are equivalent:
    \begin{enumerate} 
        \item The extension $K(\sqrt{d})/K$ is $C_4$-extendable.
        \item The element $d$ is a sum of two squares in $K$. 
        \item The element $d$ is in the norm group of the extension $K(\sqrt{-1})/K$.
    \end{enumerate}
\end{corollary}
\begin{proof}
    The equivalence of $(1)$ and $(2)$ follows from Lemma~\ref{lem-GG-norm-conditions}. If $-1 \in K^{\times 2}$, then $(2)$ and $(3)$ are equivalent because every element of $K$ can be written as a sum of two squares, due to the identity 
    $$
    d = \Big(\frac{d + 1}{2}\Big)^2 + \Big(\frac{d - 1}{2\sqrt{-1}}\Big)^2.
    $$
    If $-1 \not \in K^{\times 2}$, then the equivalence of $(2)$ and $(3)$ is trivial.
\end{proof}

By symmetry of the quadratic Hilbert symbol, it follows from Corollary~\ref{cor-extendable-iff-so2s} that we need to count extensions $K(\sqrt{d})$ such that $-1 \in \Nm K(\sqrt{d})$. Our technique for doing this applies much more generally, to counting $K(\sqrt{d})$ such that $\mathcal{A}\subseteq \Nm K(\sqrt{d})$, where $\mathcal{A}$ is any subgroup of $K^\times / K^{\times 2}$. Since it does not require much additional theory, we opt to work at this more natural level of generality. 

Let $F/K$ be an extension of $2$-adic fields. For $1 \leq t \leq v_F(2)$, write
$$
S_{F/K, t} = (U_F^{(2t)}F^{\times 2} \cap K^\times) / K^{\times 2} = \{u \in K^\times / K^{\times 2} : u/x^2 \equiv 1\pmod{\p_F^{2t}} \text{ for some $x \in F^\times$}\},
$$
and define 
$$
S_{F/K,0} = \{u\in K^\times / K^{\times 2} : v_F(u)\text{ is even}\}.
$$
For a subgroup $\mathcal{A}\subseteq K^\times / K^{\times 2}$, let $K(\sqrt{\mathcal{A}})$ be the extension 
$$
K(\{\sqrt{\alpha} : [\alpha] \in \mathcal{A}\})
$$
of $K$, write $\Nm K(\sqrt{\mathcal{A}})$ for its norm group, and define 
$$
\co_K^{\mathcal{A}} = \co_K^\times \cap \Nm K(\sqrt{\mathcal{A}}). 
$$
For $0 \leq t \leq e_K$, define the subgroup
$$
S_{K/K,t}^{\mathcal{A}} =  S_{K/K,t} \cap \Big(\Nm K(\sqrt{\mathcal{A}})/K^{\times 2}\Big).
$$
For each $m_1$, let $\Sigma_{\quadr/K,\leq m_1}^{\mathcal{A}}$ be the set of $E \in \Sigma_{\quadr/K,\leq m_1}$ with $\mathcal{A}\subseteq \Nm E$. 
\begin{lemma}
    \label{lem-bij-SKKtA-Sigma-quad-A}
    Let $0 \leq t \leq e_K$ and let $\mathcal{A}\subseteq K^\times / K^{\times 2}$ be any subgroup. We have a bijection
    $$
    S_{K/K,t}^\mathcal{A} \to \Sigma_{\quadr/K,\leq 2e_K - 2t}^{\mathcal{A}} \cup \{K\}, \quad u \mapsto K(\sqrt{u}).
    $$
\end{lemma}
\begin{proof}
    By Corollary~\ref{cor-disc-square-mod-2t+1}, the map $u\mapsto K(\sqrt{u})$ gives a well-defined bijection 
    $$
    \co_K^{\times} / \co_K^{\times 2} \to \Sigma_{\quadr/K, \leq 2e_K} \cup \{K\}. 
    $$
    For $u \in \co_K^\times\setminus \co_K^{\times 2}$, we claim that the following two statements are true:
    \begin{enumerate} 
        \item $K(\sqrt{u}) \in \Sigma_{\quadr/K,\leq 2e_K}^\mathcal{A}$ if and only if $u \in S_{K/K,0}^\mathcal{A}$.
        \item $K(\sqrt{u}) \in \Sigma_{\quadr/K, \leq 2e_K - 2t}$ if and only if $u \in S_{K/K,t}$. 
    \end{enumerate}
    The first statement follows from symmetry of the quadratic Hilbert symbol, and the second follows from Corollary~\ref{cor-disc-square-mod-2t+1}.
    The result then follows, since 
    $$
    S_{K/K,t}^\mathcal{A} = S_{K/K,0}^\mathcal{A} \cap S_{K/K,t}.
    $$
\end{proof}
For each $0 \leq t \leq e_K$, define the subgroup
$$
(\co_K/\p_K^{2t})^\mathcal{A}\subseteq (\co_K/\p_K^{2t})^\times
$$
to be the image of the map 
$$
\co_K^\mathcal{A} \to (\co_K/\p_K^{2t})^\times.
$$
\begin{lemma}
    \label{lem-SKKtA-in-terms-of-SKK0A}
    Let $0 \leq t \leq e_K$ and let $\mathcal{A}\subseteq K^\times / K^{\times 2}$ be any subgroup. There is a short exact sequence 
    $$
    1 \to S_{K/K, t}^\mathcal{A} \to S_{K/K,0}^\mathcal{A} \to \frac{(\co_K/\p_K^{2t})^\mathcal{A}}{(\co_K/\p_K^{2t})^{\times 2}} \to 1.  
    $$
\end{lemma}
\begin{proof}
    This is immediate from the definitions.
\end{proof}
\begin{lemma}
    \label{lem-index-of-OKA}
    For any subgroup $\mathcal{A}\subseteq K^\times / K^{\times 2}$, we have 
    $$
    [\co_K^\times : \co_K^\mathcal{A}] = \frac{\#\mathcal{A}}{f(K(\sqrt{\mathcal{A}})/K)}.
    $$
\end{lemma}
\begin{proof}
    Let $[\alpha_1],\ldots, [\alpha_r]\in K^\times / K^{\times 2}$ be a minimal set of generators for $\mathcal{A}$, so that 
    $$
    K(\sqrt{\mathcal{A}}) = K(\sqrt{\alpha_1},\ldots, \sqrt{\alpha_r}). 
    $$
    By class field theory, we have 
    $$
    [K^\times : \Nm K(\sqrt{\mathcal{A}})] = 2^r = \#\mathcal{A}. 
    $$
    It follows that 
    $$
    [\co_K^\times : \co_K^\mathcal{A}] = \begin{cases}
    \#\mathcal{A}\quad\text{if there exists $x \in \Nm(K(\sqrt{\mathcal{A}}))$ with $v_K(x) = 1$},
    \\
    \frac{1}{2}\cdot\#\mathcal{A} \quad\text{otherwise},
    \end{cases}
    $$
    so we need to show that there exists $x \in \Nm(K(\sqrt{\mathcal{A}}))$ with $v_K(x) = 1$ if and only if $K(\sqrt{\mathcal{A}})/K$ is totally ramified. This follows from class field theory, since $K(\sqrt{\mathcal{A}})$ contains the unramified quadratic extension $E^{\mathrm{ur}}/K$ if and only if 
    $$
    \Nm K(\sqrt{\mathcal{A}}) \subseteq \Nm E^{\mathrm{ur}} = \{x \in K^\times : 2 \mid v_K(x)\}.
    $$
\end{proof}

For each $0 \leq t \leq e_K$, let 
$$
\mathcal{A}_t = \mathcal{A} \cap \big(U_K^{(2t)}K^{\times 2}/K^{\times 2}\big). 
$$
\begin{lemma}
    \label{lem-At-tot-ram-iff-A-tot-ram}
    Let $0 \leq t \leq e_K$ and let $\mathcal{A}\subseteq K^\times / K^{\times 2}$ be any subgroup. The extension $K(\sqrt{\mathcal{A}})/K$ is totally ramified if and only if $K(\sqrt{\mathcal{A}_t})/K$ is totally ramified. 
\end{lemma}
\begin{proof}
    Suppose that $K(\sqrt{\mathcal{A}})/K$ is not totally ramified. Then $[u] \in \mathcal{A}$, where $[u] \in K^\times / K^{\times 2}$ is the unique element such that $K(\sqrt{u})/K$ is unramified. In that case, Corollary~\ref{cor-disc-square-mod-2t+1} tells us that $u \in \mathcal{A}_t$, so $K(\sqrt{\mathcal{A}_t})/K$ is not totally ramified. 

    Suppose conversely that $K(\sqrt{\mathcal{A}_t})/K$ is not totally ramified. Since $\mathcal{A}_t\subseteq \mathcal{A}$, we have $K(\sqrt{\mathcal{A}_t})\subseteq K(\sqrt{\mathcal{A}})$, so $K(\sqrt{\mathcal{A}})$ is not totally ramified. 
\end{proof}
\begin{lemma}
    \label{lem-OKAt-supremum-of-groups}
    Let $0 \leq t \leq e_K$ and let $\mathcal{A}\subseteq K^\times / K^{\times 2}$ be any subgroup. We have 
    $$
    \co_K^{\mathcal{A}_t} = \co_K^{\mathcal{A}}U_K^{(2e_K - 2t)}. 
    $$
\end{lemma}
\begin{proof}
    First we claim that 
    $$
    \Nm K(\sqrt{\mathcal{A}_t}) = U_K^{(2e_K - 2t)}\Nm K(\sqrt{\mathcal{A}}). 
    $$
    For $[\alpha] \in \mathcal{A}_t$, Corollary~\ref{cor-disc-square-mod-2t+1} tells us that $v_K(d_{K(\sqrt{\alpha})/K}) \leq 2e_K - 2t$, so $U_K^{(2e_K - 2t)}\subseteq \Nm K(\sqrt{\alpha})$, and therefore 
    $$
    U_K^{(2e_K - 2t)}\subseteq \Nm K(\sqrt{\mathcal{A}_t}). 
    $$
    Since $\mathcal{A}_t \subseteq \mathcal{A}$, class field theory tells us that 
    $$
    \Nm K(\sqrt{\mathcal{A}}) \subseteq \Nm K(\sqrt{\mathcal{A}_t}),
    $$ 
    and therefore
    $$
    U_K^{(2e_K - 2t)}\Nm K(\sqrt{\mathcal{A}}) \subseteq \Nm K(\sqrt{\mathcal{A}_t}).
    $$
    Suppose that 
    $$
    U_K^{(2e_K - 2t)}\Nm K(\sqrt{\mathcal{A}}) \subseteq G \subseteq \Nm K(\sqrt{\mathcal{A}_t}),
    $$
    for a subgroup $G$ of $K^\times$. By class field theory, there is a unique abelian extension $L/K$ such that $\Nm L = G$. We have  
    $$
    K(\sqrt{\mathcal{A}_t}) \subseteq L \subseteq K(\sqrt{\mathcal{A}}),
    $$
    so 
    $$
    L = K(\sqrt{\mathcal{B}})
    $$
    for some subgroup $\mathcal{B} \leq \mathcal{A}$. Let $[\beta] \in \mathcal{B}$. Since $U_K^{(2e_K - 2t)}\subseteq \Nm L \subseteq \Nm K(\sqrt{\beta})$, we have 
    $v_K(d_{K(\sqrt{\beta})/K}) \leq 2e_K - 2t$, so Corollary~\ref{cor-disc-square-mod-2t+1} tells us that $\beta \in U_K^{(2t)}K^{\times 2}$, and therefore $[\beta] \in \mathcal{A}_t$. It follows that $\mathcal{B}\subseteq \mathcal{A}_t$, and therefore $L \subseteq K(\sqrt{\mathcal{A}_t})$, so $G = \Nm K(\sqrt{\mathcal{A}_t})$. Therefore, as claimed, we have 
    $$
    \Nm K(\sqrt{\mathcal{A}_t}) = U_K^{(2e_K - 2t)}\Nm K(\sqrt{\mathcal{A}}). 
    $$
    It follows that 
    $$
    \co_K^{\mathcal{A}_t} = \Big(
    U_K^{(2e_K - 2t)} \Nm K(\sqrt{\mathcal{A}})     
    \Big) \cap \co_K^\times,
    $$
    so we need to show that 
    $$
    \Big(
    U_K^{(2e_K - 2t)} \Nm K(\sqrt{\mathcal{A}})     
    \Big) \cap \co_K^\times = U_K^{(2e_K - 2t)}\Big(
        \Nm K(\sqrt{\mathcal{A}})  \cap \co_K^\times
    \Big),
    $$
    which is an easy exercise in definitions.
\end{proof}
\begin{lemma}
    \label{lem-ES-for-OK/p2t-A}
    Let $1 \leq t \leq e_K$ and let $\mathcal{A}\subseteq K^\times / K^{\times 2}$ be any subgroup. There is a short exact sequence 
    $$
    1 \to \co_K^{\mathcal{A}_{e_K - t}} \to \co_K^\times \to \frac{(\co_K/\p_K^{2t})^\times}{(\co_K/\p_K^{2t})^\mathcal{A}} \to 1.
    $$
\end{lemma}
\begin{proof}
    Clearly the map $\varphi:\co_K^\times \to \frac{(\co_K/\p_K^{2t})^\times}{(\co_K/\p_K^{2t})^\mathcal{A}}$ is well-defined and surjective. It follows from the definitions that 
    $$
    \ker \varphi = \co_K^\mathcal{A} U_K^{(2t)},
    $$
    so the result follows from Lemma~\ref{lem-OKAt-supremum-of-groups}.
\end{proof}
\begin{lemma}
    \label{lem-size-of-A_eK}
    For any subgroup $\mathcal{A}\subseteq K^\times / K^{\times 2}$, we have 
    $$
    \# \mathcal{A}_{e_K} = f(K(\sqrt{\mathcal{A}})/K). 
    $$
\end{lemma}
\begin{proof}
    Lemma~\ref{lem-hecke} tells us that $U_K^{(2e_K)}K^{\times 2}/K^{\times 2} = \{1, [u]\}$, where $K(\sqrt{u})/K$ is the unique unramified quadratic extension. It follows that 
    $$
    \mathcal{A}_{e_K} = \begin{cases}
        \{[1],[u]\}\quad\text{if $K(\sqrt{u})\subseteq K(\sqrt{\mathcal{A}})$},
        \\
        \{[1]\}\quad\text{otherwise},
    \end{cases}
    $$
    and the result follows. 
\end{proof}
\begin{lemma}
    \label{lem-size-of-SKKtA}
    Let $0 \leq t \leq e_K$ and let $\mathcal{A}\subseteq K^\times / K^{\times 2}$ be any subgroup. We have
    $$
    \#S_{K/K,t}^\mathcal{A} = 2q^{e_K - t} \cdot \frac{\#\mathcal{A}_{e_K-t}}{\#\mathcal{A}}. 
    $$
\end{lemma}
\begin{proof}
    It follows from Lemma~\ref{lem-SKKtA-in-terms-of-SKK0A} that 
    $$
    \#S_{K/K,t}^\mathcal{A} = \frac{\# S_{K/K,0}^\mathcal{A} \#(\co_K/\p_K^{2t})^{\times 2}}{\#(\co_K/\p_K^{2t})^\mathcal{A}}. 
    $$
    By \cite[Proposition 3.7]{neukirch-bonn}, we have $[\co_K^\times : \co_K^{\times 2}] = 2q^{e_K}$, so the definition of $S_{K/K,0}^\mathcal{A}$ gives
    $$
    \# S_{K/K,0}^\mathcal{A} = \frac{2q^{e_K}}{[\co_K^\times : \co_K^\mathcal{A}]}.
    $$
    The result for $t=0$ then follows from Lemmas~\ref{lem-index-of-OKA} and \ref{lem-size-of-A_eK}. Now assume that $t \geq 1$. 
    Lemma~\ref{lem-ES-for-OK/p2t-A} tells us that 
    $$
    \frac{1}{\# (\co_K/\p_K^{2t})^\mathcal{A}} = \frac{[\co_K^\times : \co_K^{\mathcal{A}_{e_K-t}}]}{\#(\co_K/\p_K^{2t})^\times}.
    $$
    It follows that 
    $$
    \#S_{K/K,t}^\mathcal{A} = \frac{2q^{e_K}}{[(\co_K/\p_K^{2t})^\times : (\co_K/\p_K^{2t})^{\times 2}]}\cdot \frac{[\co_K^\times : \co_K^{\mathcal{A}_{e_K-t}}]}{[\co_K^\times : \co_K^{\mathcal{A}}]}.
    $$
    The short exact sequence 
    $$
    1 \to U_K^{(t)}/U_K^{(2t)} \overset{[u]\mapsto [u]}{\to} (\co_K/\p_K^{2t})^\times \overset{[u] \mapsto [u^2]}{\to} (\co_K/\p_K^{2t})^{\times 2} \to 1
    $$
    tells us that $[(\co_K/\p_K^{2t})^\times : (\co_K/\p_K^{2t})^{\times 2}] = q^t$. Finally, the result follows from Lemmas~\ref{lem-index-of-OKA} and \ref{lem-At-tot-ram-iff-A-tot-ram}.
\end{proof}
\begin{corollary}
    \label{cor-size-of-Sigma-quad-A-leq-m1}
    Let $0 \leq m_1 \leq 2e_K$ be an even integer and let $\mathcal{A}\subseteq K^\times / K^{\times 2}$ be any subgroup. Then
    $$
    \# \Sigma_{\quadr/K,\leq m_1}^{\mathcal{A}} = 2q^{m_1/2} \cdot \frac{\#\mathcal{A}_{m_1/2}}{\#\mathcal{A}} - 1. 
    $$
\end{corollary}
\begin{proof}
    This is immediate from Lemmas~\ref{lem-bij-SKKtA-Sigma-quad-A} and \ref{lem-size-of-SKKtA}. 
\end{proof}
\begin{corollary}
    \label{cor-num-extendables-with-even-disc}
    Let $m_1$ be an even integer with $2 \leq m_1 \leq 2e_K$. We have  
    $$
    \#\Sigma_{\quadr/K,\leq m_1}^{\extendable} =(1 + \mathbbm{1}_{m_1 \leq 2e_K - d_{(-1)}})\cdot q^{m_1/2} - 1. 
    $$
\end{corollary}
\begin{proof}
    Let $\mathcal{A} = \langle [-1] \rangle \subseteq K^\times / K^{\times 2}$. Corollary~\ref{cor-extendable-iff-so2s} tells us that
    $$
    \Sigma_{\quadr/K,\leq m_1}^{\extendable} = \Sigma_{\quadr/K,\leq m_1}^{\mathcal{A}}, 
    $$
    and it follows by Corollary~\ref{cor-size-of-Sigma-quad-A-leq-m1} that 
    $$
    \Sigma_{\quadr/K,\leq m_1}^{\extendable} = 2q^{m_1/2} \cdot \frac{\#\mathcal{A}_{m_1/2}}{\#\mathcal{A}} - 1. 
    $$
    
    Suppose first that $-1 \in K^{\times 2}$. Then $[-1] = [1]$, so $\# \mathcal{A} = \#\mathcal{A}_{m_1/2} = 1$, and the result follows since $d_{(-1)} = 0$. 

    Suppose instead that $-1 \not \in K^{\times 2}$. Then
    $$
    \#\mathcal{A} = 2,
    $$
    and (by Corollary~\ref{cor-disc-square-mod-2t+1})
    $$
    \#\mathcal{A}_{m_1/2} = 1 + \mathbbm{1}_{d_{(-1)} \leq 2e_K - m_1},
    $$
    and the result follows. 
\end{proof}
\begin{proof}[Proof of Lemma~\ref{lem-N-ext}]
    The first claim follows from the classification of quadratic extensions in \cite[Lemma 4.3]{tunnell}.
    The result for $2 \leq m_1 \leq 2e_K$ follows from Corollary~\ref{cor-num-extendables-with-even-disc}. By Lemma~\ref{lem-hecke}, for any quadratic extension $E/K$, we have $v_K(d_{E/K}) = 2e_K + 1$ if and only if $E = K(\sqrt{\alpha})$ for some $\alpha \in K^\times$ with $v_K(\alpha) = 1$. Assume that this is the case. Then Corollary~\ref{cor-extendable-iff-so2s} tells us that $E/K$ is $C_4$-extendable if and only if $\alpha$ is in the norm group of $K(\sqrt{-1})/K$, and the result follows by basic class field theory. 

\end{proof}
\begin{lemma}
    \label{lem-bound-on-d-(-1)}
    The constant $d_{(-1)}$ is an even integer with
    $$
    d_{(-1)} \leq 2\Big\lceil \frac{e_K}{2}\Big\rceil. 
    $$
\end{lemma}
\begin{proof}
    This follows from Corollary~\ref{cor-disc-square-mod-2t+1}, along with the trivial fact that 
    $$
    -1 \equiv 1 \pmod{\p_K^{e_K}}.
    $$
\end{proof}

\subsection{Counting $C_4$-extensions with a given intermediate field}

\begin{lemma}
    \label{lem-conds-for-m2-extendable}
    Let $E = K(\sqrt{d})$ be a totally ramified $C_4$-extendable extension of $K$ with $m_1 = v_K(d_{E/K})$, and let $0 \leq m_2 \leq 4e_K$ be an even integer. The following are equivalent:
    \begin{enumerate}
        \item The set $\Sigma_{\quadr/E,\leq m_2}^{C_4/K}$ is nonempty. 
        \item There is some $\beta \in \co_E^\times$ such that $\beta \equiv 1\pmod{\p_E^{4e_K - m_2}}$ and $N_{E/K}(\beta) \in dK^{\times 2}$. 
        \item We have $m_2 \geq \min\{m_1 + 2e_K, 3m_1 - 2\}$. 
    \end{enumerate}
\end{lemma}
\begin{proof}
    The first two points are equivalent by Corollary~\ref{cor-disc-square-mod-2t+1} and Lemma~\ref{lem-GG-norm-conditions}. The equivalence of (2) and (3) is essentially \cite[Proposition 3.15]{cohen-et-al}. At the start of the proof, the authors state that their ``condition $(*)$'' is equivalent to (2), and the statement of their proposition is equivalent to (3), where $t = 2e_K - \frac{m_2}{2}$. Their result is stated for prime ideals of number fields lying over $2$, but it is trivial to check that the proof works for $2$-adic fields. 
\end{proof}

\begin{lemma}
    \label{lem-SE-to-SigmaE-2-to-1}
    Let $E/K$ be a totally ramified $C_4$-extendable extension, and suppose that $0 \leq m_2 \leq 4e_K$ is an even integer such that $\Sigma_{\quadr/E, \leq m_2}^{C_4/K}$ is nonempty. Let $\omega \in E$ such that $E(\sqrt{\omega}) \in \Sigma_{\quadr/E,\leq m_2}^{C_4/K}$. Then the map 
    $$
    K^\times / K^{\times 2} \to \Sigma_{\quadr/E}^{C_4/K},\quad u \mapsto E(\sqrt{u\omega})
    $$
    is surjective and $2$-to-$1$. Moreover, this map restricts to a surjective $2$-to-$1$ map 
    $$
    S_{E/K,2e_K - \frac{m_2}{2}} \to \Sigma_{\quadr/E,\leq m_2}^{C_4/K}.
    $$
\end{lemma}
\begin{proof}
    The first claim is \cite[Proposition~1.2]{cohen-et-al}, and the second claim follows from Corollary~\ref{cor-disc-square-mod-2t+1}.
\end{proof}
Fix a totally ramified quadratic extension $E/K$ with $m_1 = v_K(d_{E/K})$, and assume that $m_1$ is even. For $0 \leq t \leq 2e_K - \frac{m_1}{2}$, define $\mathcal{Z}_{E,t}$ by the short exact sequence 
$$
1 \to S_{E/K,t} \to K^\times / K^{\times 2} \to \mathcal{Z}_{E,t} \to 1. 
$$
\begin{lemma}
    \label{lem-size-of-Z-E-t-and-S-E/K-t}
    Let $E$ be a totally ramified quadratic extension of $K$ with even discriminant exponent $m_1 = v_K(d_{E/K})$. Let $m_2$ be an even integer with $m_1 \leq m_2 \leq 4e_K$. Then we have:
    \begin{enumerate} 
    \item $$
    \# \mathcal{Z}_{E,2e_K - \frac{m_2}{2}} = \begin{cases}
        2q^{\lceil e_K - \frac{m_1 + m_2}{4}\rceil} \quad \text{if $m_2 \leq 4e_K - m_1$},
        \\
        1 \quad\text{if $m_2 > 4e_K - m_1$}.
    \end{cases}
    $$
    \item $$
    \# S_{E/K,2e_K - \frac{m_2}{2}} = \begin{cases}
        2q^{\lfloor \frac{m_1+m_2}{4}\rfloor} \quad\text{if $m_2 \leq 4e_K - m_1$},
        \\
        4q^{e_K} \quad\text{if $m_2 > 4e_K - m_1$}. 
    \end{cases}
    $$
\end{enumerate}
\end{lemma}
\begin{proof}
    \hfill
    \begin{enumerate}
    \item For $m_2 = 4e_K$, we have $\mathcal{Z}_{E,2e_K - \frac{m_2}{2}} = 1$, so we can assume that $m_1 \leq m_2 \leq 4e_K - 2$. 
    The claim is then essentially \cite[Corollary~3.13]{cohen-et-al}. Under our notation, $\mathcal{Z}_{E,t}$ corresponds\footnote{In \cite{cohen-et-al}, here are the locations of the relevant definitions: $\mathcal{Z}_{\mathfrak{C}^2}$ is defined on Page 486; $Q_K(\mathfrak{C}^2)$ is defined on Page 479; $\mathfrak{C}$ is defined on Page 478; $T$ is defined on Page 478; the angle brackets $\langle T\rangle$ denote the monoid of ideals generated by $T$ - this can be inferred from the proof of Lemma~1.6.} to Cohen, Diaz y Diaz, and Olivier's $\mathcal{Z}_{\mathfrak{P}^{2t}}$, defined in \cite[Page~486]{cohen-et-al}. As with Lemma~\ref{lem-conds-for-m2-extendable}, the statement in \cite{cohen-et-al} is for prime ideals of number fields, but the modifications to the proof are trivial. 
    \item The second claim follows from the first, together with the definition of $\mathcal{Z}_{E,t}$, and \cite[Proposition~3.7]{neukirch-bonn}. 
    \end{enumerate}
\end{proof}

\begin{corollary}
    \label{cor-size-of-Sigma_E_leq_m2}
    Let $E/K$ be a totally ramified $C_4$-extendable extension with $m_1 = v_K(d_{E/K})$ even. Let $m_2\leq 4e_K$ be an even integer and write $n_0 := \min\{m_1 + 2e_K, 3m_1 - 2\}$. Then we have 
    $$
    \# \Sigma_{\quadr/E, \leq m_2}^{C_4/K} = \begin{cases}
        0\quad\text{if $m_2 < n_0$},
        \\
        q^{\lfloor \frac{m_1 + m_2}{4}\rfloor} \quad\text{if $n_0 \leq m_2 \leq 4e_K - m_1$},
        \\
        2q^{e_K}\quad\text{if $m_2 \geq \max\{4e_K - m_1 + 2, n_0\}$}.
    \end{cases}
    $$
\end{corollary}
\begin{proof}
    Lemma~\ref{lem-conds-for-m2-extendable} deals with the case $m_2 < n_0$. Let $n_0 \leq m_2 \leq 4e_K$. By Lemma~\ref{lem-conds-for-m2-extendable}, the set $\Sigma_{\quadr/E,\leq m_2}^{C_4/K}$ is nonempty, so Lemma~\ref{lem-SE-to-SigmaE-2-to-1} tells us that 
    $$
    \# \Sigma_{\quadr/E,\leq m_2}^{C_4/K} = \frac{1}{2} \#S_{E/K,2e_K - \frac{m_2}{2}},
    $$
    and the result follows from Lemma~\ref{lem-size-of-Z-E-t-and-S-E/K-t}. 
\end{proof}
\begin{proof}[Proof of Lemma~\ref{lem-N-E-m2}]
    By \cite[Lemma~4.3]{tunnell}, either $m_1= 2e_K + 1$ or $m_1$ is even with $2 \leq m_1 \leq 2e_K$. The case where $m_1$ is even follows easily from Corollary~\ref{cor-size-of-Sigma_E_leq_m2}. For the case with $m_1$ odd, suppose that $m_1 = 2e_K + 1$. Then by Lemma~\ref{lem-hecke} we have $E = K(\sqrt{d})$ for $d\in K^\times$ with $v_K(d) = 1$. By Lemma~\ref{lem-GG-norm-conditions}, each $C_4$-extension $L/K$ extending $E$ has $L = E(\sqrt{\alpha})$ for some $\alpha \in E^\times$ with $v_K(N_{E/K}(\alpha))$ odd. It follows that $v_E(\alpha)$ is odd, so $v_E(d_{L/E}) = 4e_K + 1$ by Lemma~\ref{lem-hecke}. Therefore, 
    $$
    \Sigma_{\quadr/E}^{C_4/K} = \Sigma_{\quadr/E, 4e_K + 1}^{C_4/K}, 
    $$
    so the result follows from Lemma~\ref{lem-SE-to-SigmaE-2-to-1}.
\end{proof}
\begin{proof}[Proof of Corollary~\ref{cor-size-of-Sigma-m-C4}]
    Suppose that $L/K$ is a $C_4$-extension with intermediate quadratic field $E$. By the tower law for discriminant, we have 
    $$
    v_K(d_{L/K}) = 2v_K(d_{E/K}) + f(E/K)\cdot v_E(d_{L/E}).
    $$
    So if $L \in \Sigma_m^{C_4}$ with $m_1 = v_K(d_{E/K})$ and $m_2 = v_E(d_{L/E})$, then $m = 2m_1 + m_2$, and Lemmas~\ref{lem-N-ext} and \ref{lem-N-E-m2} tell us that either $(m_1,m_2) = (2e_K + 1, 4e_K + 1)$ or $m_1$ and $m_2$ are both even with $2 \leq m_1 \leq 2e_K$ and $4 \leq m_2 \leq 4e_K$. It follows that either $m$ is even with $8 \leq m \leq 8e_K$ or $m = 8e_K + 3$. If $m = 8e_K + 3$, then the result follows from Lemmas~\ref{lem-N-ext} and \ref{lem-N-E-m2}.

    Now consider the case where $8 \leq m \leq 8e_K$ and $m$ is even. For positive integers $m_1$ and $m_2$, write
    $
    \Sigma_{m_1,m_2}^{C_4}
    $
    for the set of totally ramified $C_4$-extensions $L/K$ such that $v_K(d_{E/K}) = m_1$ and $v_E(d_{L/E}) = m_2$. By the discussion above, we have 
    $$
    \#\Sigma_m^{C_4} = \sum_{\substack{2 \leq m_1 \leq 2e_K \\ \text{$m_1$ even}}} \# \Sigma_{m_1,m - 2m_1}^{C_4}.
    $$
    Let $2 \leq m_1 \leq 2e_K$ be even. By Lemmas~\ref{lem-N-ext} and \ref{lem-N-E-m2}, whenever $N_{\ext}(m_1) \neq 0$ we have  
    \begin{align*}
    \frac{\# \Sigma_{m_1,m - 2m_1}^{C_4}}{N_{\ext}(m_1)} &= \begin{cases}
        q^{m_1-1} \quad\text{if $m_1 = \frac{m+2}{5}$ and $m_1\leq e_K$},
        \\
        q^{\lfloor \frac{m - m_1}{4} \rfloor} - q^{\lfloor \frac{m - m_1-2}{4}\rfloor}\quad \text{if $m - 4e_K \leq m_1 \leq \min\{\frac{m}{5}, e_K\}$},
        \\
        q^{e_K}\quad\text{if $m_1 = m - 4e_K - 2$ and $m_1 \leq e_K$},
        \\
        2q^{e_K}\quad\text{if $e_K < m_1 \leq 2e_K$ and $m_1 = \frac{m - 2e_K}{3}$},
        \\
        0 \quad\text{otherwise}.
    \end{cases}
    \\
    &= \begin{cases}
        q^{\frac{m-3}{5}} \quad\text{if $m_1 = \frac{m+2}{5}$ and $8 \leq m\leq 5e_K - 2$},
        \\
        q^{\lfloor \frac{m - m_1}{4} \rfloor} - q^{\lfloor \frac{m - m_1-2}{4}\rfloor}\quad \text{if $m - 4e_K \leq m_1 \leq \min\{\frac{m}{5},e_K\}$},
        \\
        q^{e_K}\quad\text{if $m_1 = m - 4e_K - 2$ and $4e_K + 4 \leq m \leq 5e_K + 2$},
        \\
        2q^{e_K}\quad\text{if $m_1 = \frac{m - 2e_K}{3}$ and $5e_K < m \leq 8e_K$},
        \\
        0 \quad\text{otherwise}.
    \end{cases}
\end{align*}
To finish the proof, we just need to observe that
$$
q^{\lfloor \frac{m - m_1}{4}\rfloor} - q^{\lfloor \frac{m - m_1 - 2}{4}\rfloor}  = \begin{cases}
    q^{\frac{m-m_1}{4} - 1}(q - 1)\quad\text{if $m_1\equiv m\pmod{4}$},
    \\
    0 \quad\text{if $m_1 \not \equiv m \pmod{4}$}.
\end{cases}
$$
\end{proof}

\begin{proof}[Proof of Theorem~\ref{thm-size-Sigma-C4-m}]
    The possible values of $m$ come from Corollary~\ref{cor-size-of-Sigma-m-C4}. The result for $m= 8e_K + 3$ is immediate from Corollary~\ref{cor-size-of-Sigma-m-C4}. Now consider the case where $m$ is even and $8 \leq m \leq 8e_K$. The first, third, and fourth items of Corollary~\ref{cor-size-of-Sigma-m-C4} respectively are equal to 
    \begin{enumerate}
    \item $\mathbbm{1}_{8\leq m \leq 5e_K - 2}\cdot \mathbbm{1}_{m\equiv 3\pmod{5}}\cdot q^{\frac{3m-14}{10}}(1 + \mathbbm{1}_{m \leq 10e_K - 5d_{(-1)} - 2})(q - 1 - \mathbbm{1}_{m = 10e_K - 5d_{(-1)} + 8})$.
        \item $\mathbbm{1}_{4e_K + 4 \leq m \leq 5e_K + 2}\cdot q^{\frac{m}{2}-e_K - 2}(1 + \mathbbm{1}_{m \leq 6e_K - d_{(-1)} + 2})(q - 1 - \mathbbm{1}_{m = 6e_K - d_{(-1)} + 4})$.
        \item $\mathbbm{1}_{5e_K + 3 \leq m \leq 8e_K}\cdot \mathbbm{1}_{m\equiv 2e_K\pmod{3}}\cdot 2q^{\frac{m+4e_K}{6} - 1}(1 + \mathbbm{1}_{m \leq 8e_K - 3d_{(-1)}})(q - 1 - \mathbbm{1}_{m = 8e_K - 3d_{(-1)} + 6})$. 
    \end{enumerate}
    Lemma~\ref{lem-bound-on-d-(-1)} turns these into the first three points of Theorem~\ref{thm-size-Sigma-C4-m}. It remains to compute the value of 
    $$
    \sum_{\substack{\max\{2, m - 4e_K\} \leq m_1 \leq \min\{\frac{m}{5}, e_K\} \\ m_1 \equiv m \pmod{4}}} q^{\frac{m-m_1}{4} - 1}(q-1)N_{\ext}(m_1).
    $$
    For such $m_1$, we have
    $$
    N_{\ext}(m_1) = \begin{cases}
        2q^{\frac{m_1}{2} - 1}(q-1) \quad\text{if $m_1 \leq 2e_K - d_{(-1)}$},
        \\
        q^{\frac{m_1}{2} - 1}(q-2)\quad\text{if $m_1 = 2e_K - d_{(-1)} + 2$},
        \\
        q^{\frac{m_1}{2} - 1}(q-1) \quad\text{if $m_1 \geq 2e_K - d_{(-1)} + 4$}.
    \end{cases}
    $$
    Lemma~\ref{lem-bound-on-d-(-1)} tells us that $2e_K - d_{(-1)} + 2 > e_K$, so the sum is actually 
    $$
    \sum_{\substack{\max\{2, m - 4e_K\} \leq m_1 \leq \min\{\frac{m}{5}, e_K
     \} \\ m_1 \equiv m \pmod{4}}} 2 q^{\frac{m + m_1}{4} - 2}(q-1)^2.
    $$
    For integers $l$ and $u$, the substitution $m_1 = -m + 4k$ makes it easy to see that
    $$
    \sum_{\substack{l \leq m_1 \leq u \\ m_1 \equiv m\pmod{4}}} q^{\frac{m + m_1}{4}} = \mathbbm{1}_{l \leq u}\cdot \frac{q^{b + 1} - q^a}{q-1},
    $$
    where $a = \lceil \frac{m+l}{4}\rceil$ and $b = \lfloor \frac{m+u}{4}\rfloor$. In this case, we have $l = \max\{2, m - 4e_K\}$ and $u = \min\{e_K, \frac{m}{5}\}$, which gives 
    $$
    a = \lceil \max\{
        \frac{m+2}{4}
        ,
        \frac{m}{2} - e_K
    \}\rceil,\quad b = \lfloor \min\{
        \frac{m+e_K}{4}
        ,
        \frac{3m}{10}
    \}\rfloor.
    $$
    Finally, it is easy to see that $l \leq u$ if and only if $10 \leq m \leq 5e_K$. In that case, we have $b = \lfloor \frac{3m}{10}\rfloor$, so 
    $$
    \sum_{\substack{\max\{2, m - 4e_K\} \leq m_1 \leq \min\{e_K, \frac{m}{5}
     \} \\ m_1 \equiv m \pmod{4}}} q^{\frac{m + m_1}{4}} =  \mathbbm{1}_{10 \leq m \leq 5e_K} \cdot \frac{
        q^{\lfloor \frac{3m}{10}\rfloor + 1}
        - 
        q^{\lceil \max\{
            \frac{m+2}{4}
            ,
            \frac{m}{2} - e_K
        \}\rceil}
     }{q-1},
    $$
    and the result follows.
\end{proof}
\begin{proof}[Proof of Corollary~\ref{cor-premass-of-Sigma-C4}]
    Theorem~\ref{thm-size-Sigma-C4-m} and Lemma~\ref{lem-bound-on-d-(-1)} tell us that the mass is the sum of the following quantities:
    \begin{enumerate}
        \item  $$\frac{1}{2}\cdot \sum_{\substack{8 \leq m \leq 5e_K - 2 \\ m \equiv 8\pmod{10}}} q^{-\frac{7m + 14}{10}}(q-1).$$
        \item $$
                \frac{1}{2}\cdot \sum_{\substack{
                    4e_K + 4 \leq m \leq 5e_K + 2
                    \\
                    \text{$m$ even}
                }} q^{-\frac{m}{2} - e_K - 2}(q-1).
            $$
        \item \begin{enumerate}
            \item $$
                \sum_{\substack{5e_K + 3 \leq m \leq 8e_K - 3d_{(-1)} \\ m \equiv 2e_K \pmod{6}}} q^{\frac{4e_K-5m}{6} - 1}(q-1). 
            $$
            \item $$
                \mathbbm{1}_{d_{(-1)} \geq 2}\cdot \frac{1}{2}\cdot q^{-6e_K + \frac{5}{2}d_{(-1)} - 6}(q-2).     
            $$
            \item $$
                \frac{1}{2}\cdot \sum_{\substack{8e_K - 3d_{(-1)} + 12 \leq m \leq 8e_K \\ m\equiv 2e_K\pmod{6}}}  q^{\frac{4e_K-5m}{6} - 1}(q-1).   
            $$
        \end{enumerate}
        \item \begin{enumerate}
            \item $$
            \frac{1}{2} (q-1)q^{-1}\sum_{\substack{10 \leq m \leq 5e_K \\ \text{$m$ even}}} q^{\lfloor-\frac{7m}{10}\rfloor
                }. 
            $$
            \item $$
            - \frac{1}{2} (q-1)q^{-2}\sum_{\substack{10 \leq m \leq 5e_K \\ \text{$m$ even}}} q^{\max\{
                \lceil \frac{-3m+2}{4}\rceil,
                -\frac{m}{2} - e_K
            \}}.
            $$
        \end{enumerate}
        \item $$
        \begin{cases}
            q^{-6e_K - 3} \quad\text{if $-1 \in K^{\times 2}$},
            \\
            \frac{1}{2}q^{-6e_K-3}\quad\text{if $K(\sqrt{-1})/K$ is quadratic and totally ramified},
            \\
            0 \quad\text{otherwise}. 
        \end{cases}
    $$
    \end{enumerate}
    We address these one by one.
    \begin{enumerate}
        \item Making the substitution $m = 10k + 8$, we have 
            \begin{align*}
            \sum_{\substack{8 \leq m \leq 5e_K - 2 \\ m \equiv 8\pmod{10}}} & q^{-\frac{7m + 14}{10}} = \sum_{k=0}^{\lfloor \frac{e_K}{2}\rfloor - 1} q^{-7k - 7}
            \\
            &= \mathbbm{1}_{e_K \geq 2} \cdot \frac{1 - q^{- 7 \lfloor \frac{e_K}{2}\rfloor}}{q^7-1},
            \end{align*}
            so the contribution to the mass is 
            $$
            \frac{1}{2}\cdot \mathbbm{1}_{e_K \geq 2} \cdot \frac{(q-1)(1 - q^{-7\lfloor \frac{e_K}{2}\rfloor})}{q^7 - 1},
            $$
            and we can omit the indicator function since $e_K = 1$ gives $1 - q^{-7\lfloor\frac{e_K}{2}\rfloor}=0$.
        \item Making the substitution $m = 2k$, it is easy to see that 
            $$
            \sum_{\substack{
                    4e_K + 4 \leq m \leq 5e_K + 2
                    \\
                    \text{$m$ even}
                }} q^{-\frac{m}{2}} = \mathbbm{1}_{e_K \geq 2} \cdot \frac{q^{
                    -2e_K - 1
                } - q^{
                    - \lfloor \frac{5e_K + 2}{2}\rfloor
                }}{q-1},
            $$
            so the contribution is 
            $$
            \frac{1}{2} \cdot(q^{
                -3e_K - 3
            } - q^{
                - \lfloor \frac{7e_K + 6}{2}\rfloor
            }) = \frac{1}{2}\cdot q^{-3e_K - 3}(1 - q^{
                - \lfloor \frac{e_K}{2}\rfloor
            }),
            $$ 
            where we omit the indicator function since $e_K = 1$ gives $q^{-2e_K - 1} - q^{-\lfloor \frac{5e_K +2}{2}\rfloor} = 0$.
        \item \begin{enumerate}
            \item The substitution $m = 2e_K + 6k$ gives 
            \begin{align*}
            \sum_{\substack{5e_K + 3 \leq m \leq 8e_K - 3d_{(-1)} \\ m\equiv 2e_K \pmod{6}}} q^{\frac{4e_K - 5m}{6}} &= \sum_{k=\lfloor \frac{e_K}{2}\rfloor + 1}^{e_K - \frac{1}{2}d_{(-1)}} q^{-e_K - 5k}
            \\
            &= \mathbbm{1}_{d_{(-1)} < e_K}\frac{q^{-5\lfloor\frac{e_K}{2}\rfloor - e_K} - q^{\frac{5}{2}d_{(-1)} - 6e_K}}{q^5-1},
            \end{align*}
            so the contribution is 
            $$
            \mathbbm{1}_{d_{(-1)} < e_K} \cdot \frac{(q-1)(q^{-5\lfloor\frac{e_K}{2}\rfloor - e_K - 1} - q^{\frac{5}{2}d_{(-1)} - 6e_K - 1})}{q^5-1}.
            $$
            \item This is already in closed form. 
            \item The substitution $m = 2e_K + 6k$ gives
            \begin{align*}
                \sum_{\substack{8e_K - 3d_{(-1)} + 12\leq m \leq 8e_K \\ m \equiv 2e_K \pmod{6}}} q^{\frac{4e_K - 5m}{6}} &= \sum_{k=e_K - \frac{1}{2}d_{(-1)} + 2}^{e_K}q^{-e_K - 5k}
                \\
                &=\mathbbm{1}_{d_{(-1)} \geq 4}\cdot \frac{q^{\frac{5}{2}d_{(-1)} - 6e_K-5} - q^{-6e_K}}{q^5 - 1}. 
            \end{align*}
            Therefore, the contribution is 
            $$
            \frac{1}{2} \cdot \mathbbm{1}_{d_{(-1)} \geq 4}\cdot \frac{(q-1)(q^{\frac{5}{2}d_{(-1)} - 6e_K - 6} - q^{-6e_K-1})}{q^5-1}.
            $$
        \end{enumerate}
        \item \begin{enumerate}
            \item We need to compute 
            $$
            \sum_{\substack{10 \leq m \leq 5e_K \\ \text{$m$ even}}} q^{\lfloor \frac{-7m}{10}\rfloor} = \sum_{k=5}^{\lfloor \frac{5e_K}{2}\rfloor} q^{- \lceil \frac{7k}{5}\rceil}.
            $$
            Let $b\geq 1$ be an integer, and consider the sum 
            \begin{align*}
            \sum_{k=5}^{5b} q^{-\lceil \frac{7k}{5}\rceil} &= \sum_{a=1}^{b-1}\Big(
                \sum_{k=5a}^{5a+4}q^{-\lceil \frac{7k}{5}\rceil} 
            \Big) + q^{-7b}
            \\
            &= \sum_{a=1}^{b-1}q^{-7a}\Big(
                \sum_{l=0}^4q^{-\lceil \frac{7l}{5}\rceil}
            \Big) + q^{-7b}
            \\
            &= q^{-6}(q^6 + q^4 + q^3 + q + 1) \cdot \frac{(1 - q^{7-7b})}{q^7-1} + q^{-7b}
            \\
            &= \frac{(q^{-6} - q^{1-7b})(q^6 + q^4 + q^3 + q + 1)}{q^7-1} + q^{-7b}.
            \end{align*}
            Suppose that $e_K$ is even. Then we have 
            \begin{align*}
                \sum_{k=5}^{\lfloor \frac{5e_K}{2}\rfloor} q^{- \lceil \frac{7k}{5}\rceil} &= \sum_{k=5}^{5\cdot\frac{e_K}{2}} q^{-\lceil\frac{7k}{5}\rceil} 
                \\
                &= \mathbbm{1}_{e_K\geq 2}\cdot \Big(\frac{(q^{-6} - q^{1-\frac{7e_K}{2}})(q^6 + q^4 + q^3 + q + 1)}{q^7-1} + q^{-\frac{7e_K}{2}}\Big).
            \end{align*}
            Suppose instead that $e_K$ is odd. Then we have 
                \begin{align*}
                    \sum_{k=5}^{\lfloor \frac{5e_K}{2}\rfloor} q^{- \lceil \frac{7k}{5}\rceil} &=  \sum_{k=5}^{\frac{5e_K - 1}{2}} q^{- \lceil \frac{7k}{5}\rceil}
                    \\
                    &= \sum_{k=5}^{5\cdot\frac{e_K-1}{2}}q^{-\lceil \frac{7k}{5}\rceil} + q^{-\lceil\frac{7}{5}\cdot \frac{5e_K - 3}{2}\rceil} + q^{-\lceil\frac{7}{5}\cdot \frac{5e_K-1}{2}\rceil}
                    \\
                    &= \mathbbm{1}_{e_K \geq 2}\cdot \Big(\frac{(q^{-6} - q^{1 - 7\cdot\frac{e_K - 1}{2}})(q^6 + q^4 + q^3 + q + 1)}{q^7-1} + q^{-7\cdot \frac{e_K-1}{2}} 
                    \\
                    &\quad \quad\quad + q^{-7\cdot \frac{e_K-1}{2} - 2} + q^{-7\cdot\frac{e_K-1}{2} - 3}\Big).
                \end{align*}
            In other words, the sum $\sum_{k=5}^{\lfloor \frac{5e_K}{2}\rfloor} q^{- \lceil \frac{7k}{5}\rceil}$ is equal to 
            $$
            \mathbbm{1}_{e_K \geq 2}\cdot\Big( \frac{(q^{-6} - q^{1 - 7\lfloor\frac{e_K}{2}\rfloor})(q^6 + q^4 + q^3 + q + 1)}{q^7 - 1} + q^{-7\lfloor\frac{e_K}{2}\rfloor}(1 + \mathbbm{1}_{2\nmid e_K}(q^{-2} + q^{-3}))\Big).
            $$
            Therefore we have a contribution of 
            $$
            \mathbbm{1}_{e_K\geq 2}\cdot \frac{1}{2}(q-1)\Big(
                \frac{(q^{-7} - q^{ - 7\lfloor\frac{e_K}{2}\rfloor})(q^6 + q^4 + q^3 + q + 1)}{q^7 - 1} + q^{-7\lfloor\frac{e_K}{2}\rfloor-1}(1 + \mathbbm{1}_{2\nmid e_K}(q^{-2} + q^{-3}))
            \Big).
            $$
            \item We need to evaluate 
            \begin{align*}
            \sum_{k=5}^{\lfloor \frac{5e_K}{2}\rfloor} q^{\max\{
                \lceil \frac{-3k+1}{2}\rceil, -k - e_K    
            \}} &= \sum_{k=5}^{2e_K}q^{\lceil \frac{-3k+1}{2}\rceil} + \sum_{k=2e_K + 1}^{\lfloor\frac{5e_K}{2}\rfloor} q^{-k-e_K}.
            \end{align*}
            We have 
            \begin{align*}
            \sum_{k=5}^{2e_K} q^{\lceil  \frac{-3k+1}{2}\rceil} &= \sum_{a=3}^{e_K}\sum_{k=2a-1}^{2a}q^{\lceil \frac{-3k+1}{2}\rceil} 
            \\
            &= \sum_{a=3}^{e_K}(q^{-3a + 2} + q^{-3a+1})
            \\
            &= (q^2+q)\sum_{a=3}^{e_K}q^{-3a}
            \\
            &= \mathbbm{1}_{e_K \geq 3}\cdot \frac{(q^2+q)(q^{-6} - q^{-3e_K})}{q^3-1}. 
            \end{align*}
            So the first half of the sum gives a contribution of 
            $$
            -\mathbbm{1}_{e_K \geq 2}\cdot \frac{1}{2}\cdot\frac{(q-1)(q+1)(q^{-7} - q^{-3e_K - 1})}{q^3-1}.
            $$
            We also have 
            $$
            \sum_{k=2e_K+1}^{\lfloor\frac{5e_K}{2}\rfloor} q^{-k-e_K} = \mathbbm{1}_{e_K\geq 2}\cdot \frac{q^{-3e_K} - q^{-\lfloor\frac{5e_K}{2}\rfloor - e_K}}{q-1},
            $$
            so we also get a contribution of 
            $$
            - \frac{1}{2}(q^{-3e_K - 2} - q^{-\lfloor\frac{7e_K}{2}\rfloor - 2}).
            $$
        \end{enumerate}
    \end{enumerate}
\end{proof}
\section{The Case $G = D_4$}
\label{sec-G=D4}

For $G \in \{V_4, C_4, D_4\}$ and $L \in \Sigma^G$, let 
$$
\Sigma_{\quadr/K}^{\hookrightarrow L} := \{E \in \Sigma_{\quadr/K} : \exists K\text{-morphism } E \hookrightarrow L\}.
$$
Let $L \in \Sigma^G$ and $E \in \Sigma_{\quadr/K}^{\hookrightarrow L}$. There is a unique embedding $E \hookrightarrow L$, so we may naturally view $L$ as an extension of $E$. We define a \emph{$K$-twist of $L/E$} to be an element of the set
$$
\Twist_K(L/E) = \{L' \in \Sigma_{\quadr/E} : \exists K\text{-isomorphism } L' \cong L\}.
$$

\begin{lemma}
    \label{lem-size-of-Xi}
    Let $G \in \{V_4, C_4, D_4\}$. The following two statements are true:
    \begin{enumerate}
        \item For $L \in \Sigma^G$, we have
        $$
        \# \Sigma_{\quadr/K}^{\hookrightarrow L} = \begin{cases}
            1 \quad\text{if $G \in \{C_4, D_4\}$},
            \\
            3\quad\text{if $G = V_4$}.
        \end{cases}
        $$
        \item For $L \in \Sigma^G$ and $E \in \Sigma_{\quadr/K}^{\hookrightarrow L}$, we have 
        $$
        \# \Twist_K(L/E) = \begin{cases}
            1 \quad\text{if $G \in \{C_4, V_4\}$},
            \\
            2 \quad\text{if $G = D_4$}. 
        \end{cases}
        $$
    \end{enumerate}
\end{lemma}
\begin{proof}
    Claim (1) is obvious. For Claim (2), write $E = K(\sqrt{d})$ and $L = E(\sqrt{\alpha})$, where $d \in K$ and $\alpha \in E$. Let $L' \in \Twist_K(L/E)$. Then there is a $K$-isomorphism $\varphi: E(\sqrt{\alpha}) \to L'$. We will view $E$ as a subset of both extensions $L$ and $L'$, even though $L$ and $L'$ are not necessarily inside the same algebraic closure of $E$. 
    
    The element $\varphi(\sqrt{\alpha})\in L'$ has the same minimal polynomial over $K$ as $\sqrt{\alpha} \in L$, so either $L' \cong E(\sqrt{\alpha})$ or $L' \cong E(\sqrt{\overline{\alpha}})$, where $\overline{\alpha}$ is the conjugate of $\alpha$ over $K$. It is easy to see that both these choices for $L'$ are in $\Twist_K(L/E)$, so 
    $$
    \Twist_K(L/E) = \{E(\sqrt{\alpha}), E(\sqrt{\overline{\alpha}})\}.
    $$
    By elementary Galois theory, we have $E(\sqrt{\alpha})\not \cong E(\sqrt{\overline{\alpha}})$ over $E$ if and only if $G = D_4$. 
\end{proof}
For an integer $m$, define an \emph{$m$-tower} to be a pair $(E,L)$, where $E \in \Sigma_{\quadr/K}$ and $L \in \Sigma_{\quadr/E}$, such that $L/K$ is a totally ramified extension with $v_K(d_{L/K}) = m$. Write $\Tow_m$ for the set of $m$-towers. There is a natural surjection 
$$
\Phi_m : \Tow_m \to \Sigma_m^{C_4} \cup \Sigma_m^{V_4} \cup \Sigma_m^{D_4},\quad (E,L)\mapsto L. 
$$
\begin{lemma}
    \label{lem-size-of-fibre-of-Phi_m}
    Let $G \in \{C_4, V_4, D_4\}$, let $m$ be an integer, and let $L_0 \in \Sigma_m^G$. The fibre $\Phi_m^{-1}(L_0)$ has size 
    $$
    \begin{cases}
        1\quad\text{if $G = C_4$},
        \\
        2\quad\text{if $G = D_4$},
        \\
        3 \quad\text{if $G = V_4$}. 
    \end{cases}
    $$
\end{lemma}
\begin{proof}
    It is easy to see that 
    $$
    \Phi_m^{-1}(L_0) = \{(E,L) : E \in \Sigma_{\quadr/K}^{\hookrightarrow L_0}, L \in \Twist_K(L_0/E)\},
    $$
    and the result follows from Lemma~\ref{lem-size-of-Xi}.
\end{proof}

\begin{corollary}
    \label{cor-Xm-in-terms-of-field-sets}
    For every integer $m$, we have 
    $$
    \# \Sigma_m^{C_4} + 2\cdot \#\Sigma_m^{D_4} + 3\cdot \#\Sigma_m^{V_4} = \# \Tow_m. 
    $$
\end{corollary}
\begin{proof}
    This is immediate from Lemma~\ref{lem-size-of-fibre-of-Phi_m}.
\end{proof}
\begin{lemma}
    \label{lem-size-of-Xm}
    If $\Tow_m$ is nonempty, then one of the following three statements is true: 
    \begin{enumerate}
        \item $m$ is an even integer with $6  \leq m \leq 8e_K + 2$.
        \item $m\equiv 1 \pmod{4}$ and $4e_K + 5 \leq m \leq 8e_K + 1$. 
        \item $m = 8e_K + 3$. 
    \end{enumerate}
    For even $m$ with $6 \leq m \leq 8e_K + 2$, we have 
    $$
    \# \Tow_m = 4(q-1)q^{\frac{m}{2} - 2}(\mathbbm{1}_{m \geq 4e_K + 4}\cdot q^{-e_K} +\mathbbm{1}_{m \leq 8e_K}\cdot ( q^{\min\{0, e_K + 1 - \lceil \frac{m}{4}\rceil\} }- q^{-\min\{\lfloor \frac{m-2}{4}\rfloor, e_K\}})).
    $$
    For $m \equiv 1\pmod{4}$ with $4e_K + 5 \leq m \leq 8e_K + 1$, we have 
    $$
    \# \Tow_m = 4(q-1)q^{e_K + \frac{m-1}{4} - 1}. 
    $$
    We also have 
    $$
    \# \Tow_{8e_K+3} = 4q^{3e_K}. 
    $$
\end{lemma}
\begin{proof}
    Let $m$ be an integer such that $\Tow_m$ is nonempty. Let $(E,L) \in \Tow_m$, and let $m_1 = v_K(d_{E/K})$ and $m_2 = v_E(d_{L/E})$, so that $m = 2m_1 + m_2$ by the tower law for discriminant. By \cite[Lemma~4.3]{tunnell}, either $m_1$ is even with $2 \leq m_1 \leq 2e_K$, or $m_1 = 2e_K + 1$. Similarly, either $m_2$ is even with $2 \leq m_2 \leq 4e_K$, or $m_2 = 4e_K + 1$. If $m_2$ is even, then $m$ is even and $6 \leq m \leq 8e_K + 2$. If $m_2 = 4e_K + 1$ and $m_1$ is even, then $m \equiv 1\pmod{4}$ and $4e_K + 5 \leq m \leq 8e_K + 1$. Finally, if $m_1$ and $m_2$ are both odd, then $m = 8e_K + 3$. Now that we have identified the possibilities, we can enumerate $\Tow_m$ in each case.
    
    Suppose first that $m$ is even with $6 \leq m \leq 8e_K + 2$. Then each $(E,L) \in \Tow_m$ has $m_2$ even, so $\#\Tow_m$ is the sum of the following two quantities:
    \begin{enumerate}
    \item $$\sum_{\substack{\max\{2, \frac{m}{2} - 2e_K\} \leq m_1 \leq \min\{\frac{m}{2} - 1, 2e_K\} \\ \text{$m_1$ even}}} \sum_{E \in \Sigma_{\quadr/K,m_1}} \#\Sigma_{\quadr/E, m - 2m_1}.$$
    \item $$\mathbbm{1}_{m \geq 4e_K + 4}\cdot \sum_{E \in \Sigma_{\quadr/K,2e_K + 1}} \#\Sigma_{\quadr/E,m - 4e_K - 2}.
    $$
    \end{enumerate}
    By \cite[Lemma~4.3]{tunnell}, the first of these quantities is equal to 
    \begin{align*}
        \#\Sigma_{\quadr/E, m - 2m_1} 
        &= \sum_{\substack{\max\{2, \frac{m}{2} - 2e_K\} \leq m_1 \leq \min\{\frac{m}{2} - 1, 2e_K\} \\ \text{$m_1$ even}}} 4(q-1)^2 q^{\frac{m-m_1}{2} - 2}
        \\
        &= 4(q-1)^2q^{\frac{m}{2} - 2}\sum_{k=a}^b q^{-k}
        \\
        &= 4(q-1)^2q^{\frac{m}{2} - 2} \cdot \mathbbm{1}_{a \leq b}\cdot \frac{q^{1-a} - q^{-b}}{q-1}
        \\
        &= \mathbbm{1}_{6 \leq m \leq 8e_K} \cdot 4(q-1)q^{\frac{m}{2} - 2}(q^{1-a} - q^{-b}),
        \end{align*}
        where 
        $$
        a := \max\{1, \lceil\frac{m}{4}\rceil - e_K\}, \quad b := \min\{\lfloor \frac{m-2}{4}\rfloor, e_K\}. 
        $$
    For $m = 2,4$ we have $q^{1-a} - q^{-b} = 0$, so we may drop the ``$6 \leq m$'' from the indicator function, giving 
    $$
    \#\Sigma_{\quadr/E, m - 2m_1} = \mathbbm{1}_{m \leq 8e_K} \cdot 4(q-1)q^{\frac{m}{2} - 2}(q^{1-a} - q^{-b}).
    $$
    Similarly, the second quantity is equal to
    $$
    \mathbbm{1}_{m \geq 4e_K + 4}\cdot 4(q-1) q^{\frac{m}{2} - e_K - 2},
    $$
    and we obtain the desired expression for $\# \Tow_m$. Now suppose that $m\equiv 1\pmod{4}$ and $4e_K + 5 \leq m \leq 8e_K + 1$. Then each $(E,L) \in \Tow_m$ has $m_2 = 4e_K + 1$ and $m_1 = \frac{m-1}{2} - 2e_K$, so \cite[Lemma~4.3]{tunnell} gives us
    \begin{align*}
    \# \Tow_m &= \sum_{E \in \Sigma{\quadr/K,\frac{m-1}{2}-2e_K}} \#\Sigma_{\quadr/E,4e_K+1}
    \\
    &= 4(q-1)q^{e_K + \frac{m-1}{4} - 1}.
    \end{align*}
    Finally, if $m = 8e_K + 3$, then each $(E,L) \in \Tow_m$ has $m_1 = 2e_K + 1$ and $m_2 = 4e_K + 1$, so 
    \begin{align*}
    \# \Tow_m &= \sum_{E\in \Sigma_{\quadr/K,2e_K+1}} \#\Sigma_{E,4e_K + 1}
    \\
    &= 4q^{3e_K},
    \end{align*}
    by \cite[Lemma~4.3]{tunnell}.
\end{proof}

\begin{proof}[Proof of Theorem~\ref{thm-size-of-Sigma_m_D4}]
    This is immediate from Corollary~\ref{cor-Xm-in-terms-of-field-sets} and Lemma~\ref{lem-size-of-Xm}. 
\end{proof}
\begin{lemma}
    \label{lem-"mass"-of-towers}
    We have 
    $$
    \frac{1}{4}\sum_m q^{-m}\#\Tow_m = \frac{1}{q^2+q+1}(q^{-3e_K - 3} + q^{-3e_K - 1} + q^{-2}).
    $$
\end{lemma}
\begin{proof}
    Lemma~\ref{lem-size-of-Xm} tells us that $
    \frac{1}{4}\sum_m q^{-m}\#\Tow_m 
    $ is the sum of the following four quantities:
    \begin{enumerate}
        \item $\sum_{\substack{4e_K + 4 \leq m \leq 8e_K + 2 \\ \text{$m$ even}}} (q-1)q^{-\frac{m}{2} - e_K - 2}$.
        \item $\sum_{\substack{6 \leq m \leq 8e_K \\ \text{$m$ even}}} (q-1)q^{\min\{0, e_K + 1 - \lceil \frac{m}{4}\rceil\} -\frac{m}{2} - 2}$.
        \item $- \sum_{\substack{6 \leq m \leq 8e_K \\ \text{$m$ even}}}(q-1)q^{- \frac{m}{2} - 2 - \min\{\lfloor \frac{m-2}{4}\rfloor, e_K\}}$.
        \item $\sum_{\substack{4e_K + 5 \leq m \leq 8e_K + 1 \\ m\equiv 1\pmod{4}}} (q-1)q^{e_K + \frac{-3m-1}{4} - 1}$.
        \item $q^{-5e_K - 3}$.
    \end{enumerate}
    We can simplify this as the sum of the following quantities:
    \begin{enumerate}
        \item $(q-1)q^{-e_K - 2}\cdot \sum_{\substack{4e_K + 4 \leq m \leq 8e_K + 2\\ \text{$m$ even}}}q^{-\frac{m}{2}}$.
        \item \begin{enumerate}
            \item $(q-1)q^{-2}\cdot\sum_{\substack{6 \leq m \leq 4e_K + 4 \\ \text{$m$ even}}} q^{-\frac{m}{2}}$.
            \item $(q-1)q^{e_K - 1}\cdot\sum_{\substack{4e_K + 6 \leq m \leq 8e_K \\ \text{$m$ even}}}q^{-\lceil \frac{3m}{4}\rceil}$.
        \end{enumerate}
        \item \begin{enumerate}
            \item $-(q-1)q^{-e_K-2}\cdot\sum_{\substack{4e_K + 2 \leq m \leq 8e_K \\ \text{$m$ even}}} q^{-\frac{m}{2}}$.
            \item $-(q-1)q^{-2}\cdot\sum_{\substack{6 \leq m \leq 4e_K \\ \text{$m$ even}}} q^{-\lfloor\frac{3m-2}{4}\rfloor}$. 
        \end{enumerate}
        \item $(q-1)q^{e_K - 1}\cdot \sum_{\substack{4e_K + 5 \leq m \leq 8e_K + 1 \\ m\equiv 1\pmod{4}}} q^{-\frac{3m + 1}{4}}$.
        \item $q^{-5e_K - 3}$.
    \end{enumerate}
    A further round of simplifications yields 
    \begin{enumerate}
        \item $(q-1)q^{-e_K - 2}\cdot \sum_{k=2e_K + 2}^{4e_K+1}q^{-k}$.
        \item \begin{enumerate}
            \item $(q-1)q^{-2}\cdot\sum_{k=3}^{2e_K+2}q^{-k}$.
            \item $(q-1)q^{e_K - 1}\cdot\sum_{k=2e_K + 3}^{4e_K}q^{-\lceil\frac{3k}{2}\rceil}$.
        \end{enumerate}
        \item \begin{enumerate}
            \item $-(q-1)q^{-e_K-2}\cdot\sum_{k=2e_K+1}^{4e_K}q^{-k}$.
            \item $-(q-1)q^{-2}\cdot\sum_{k=3}^{2e_K}q^{-\lfloor \frac{3k-1}{2}\rfloor}$.
        \end{enumerate}
        \item $(q-1)q^{e_K-2}\cdot \sum_{k=e_K + 1}^{2e_K}q^{-3k}$.
        \item $q^{-5e_K - 3}$.
    \end{enumerate}
    We put the pieces together to obtain the contributions to the final sum:
    \begin{itemize}
        \item (1) and (3)(a) cancel to give a contribution of
        $$
        (q-1)(q^{-5e_K - 3} - q^{-3e_K - 3}).
        $$
        \item (2)(a) simplifies to a contribution of  
        $$
        q^{-4} - q^{-2e_K - 4}.
        $$
        \item We have 
        $$
        \sum_{k=2e_K + 3}^{4e_K}q^{-\lceil\frac{3k}{2}\rceil} = \frac{q+1}{q^3-1}(q^{-3e_K - 3} - q^{-6e_K}),
        $$
        so (2)(b) gives a contribution of 
        $$
        \frac{q+1}{q^2+q+1}(q^{-2e_K - 4} - q^{-5e_K - 1}).
        $$
        \item We have 
        $$
        \sum_{k=3}^{2e_K} q^{-\lfloor\frac{3k-1}{2}\rfloor} = \frac{q+1}{q^3-1}(q^{-2} - q^{1-3e_K}),
        $$
        so $(3)(b)$ gives a contribution of 
        $$
        -\frac{q+1}{q^2+q+1}(q^{-4} - q^{-1-3e_K}).
        $$
        \item We have 
        $$
        \sum_{k=e_K+1}^{2e_K} q^{-3k} = \frac{q^{-3e_K} - q^{-6e_K}}{q^3-1},
        $$
        so (4) gives a contribution of 
        $$
        \frac{1}{q^2+q+1}(q^{-2e_K - 2} - q^{-5e_K - 2}).
        $$
        \item Finally, (5) obviously gives a contribution of 
        $$
        q^{-5e_K - 3}.
        $$
    \end{itemize}
    So far, we have shown that $\frac{1}{4}\sum_m q^{-m}\#\Tow_m$ is the sum of the following six quantities:
    \begin{enumerate}
        \item[(A)] $(q-1)(q^{-5e_K - 3} - q^{-3e_K - 3})$.
        \item[(B)] $q^{-4} - q^{-2e_K - 4}$.
        \item[(C)] $\frac{q+1}{q^2+q+1}(q^{-2e_K - 4} - q^{-5e_K - 1})$.
        \item[(D)] $-\frac{q+1}{q^2+q+1}(q^{-4} - q^{-3e_K - 1})$.
        \item[(E)] $\frac{1}{q^2+q+1}(q^{-2e_K-2} - q^{-5e_K-2})$.
        \item[(F)] $q^{-5e_K - 3}$.
    \end{enumerate}
    The sum of $(C),(D)$ and $(E)$ is 
    $$
    q^{-2e_K-4} - q^{-5e_K-2} + \frac{q+1}{q^2+q+1}(q^{-3e_K - 1} - q^{-4}),
    $$
    so we have shown that $\sum_m q^{-m}\# \Tow_m$ is the sum of the following four quantities:
    \begin{enumerate}
        \item $(q-1)(q^{-5e_K - 3} - q^{-3e_K - 3})$.
        \item $q^{-4} - q^{-2e_K - 4}$.
        \item $q^{-2e_K-4} - q^{-5e_K-2} + \frac{q+1}{q^2+q+1}(q^{-3e_K - 1} - q^{-4})$.
        \item $q^{-5e_K - 3}$.
    \end{enumerate}
    It is easy to check that this sum simplifies to 
    $$
    \frac{1}{q^2+q+1}(q^{-3e_K - 3} + q^{-3e_K - 1} + q^{-2}),
    $$
    so we are done. 
\end{proof}

\begin{proof}[Proof of Corollary~\ref{cor-premass-of-D4}]
    This follows easily from Corollary~\ref{cor-Xm-in-terms-of-field-sets}, Lemma~\ref{lem-"mass"-of-towers}, and the definition of mass.
\end{proof}

\printbibliography
\end{document}